 \documentclass[11pt]{amsart}
\usepackage{amsthm,amsmath,amssymb,amscd,graphicx,enumerate, stmaryrd,xspace,verbatim, epic, eepic,color,url}

\usepackage{eurosym}
\usepackage[all]{xypic}
\SelectTips{cm}{}

   \newtheorem{thm}{Theorem}[subsection]
      \newtheorem*{thm*}{Theorem}

   \newtheorem{prop}[thm] {Proposition}     
   \newtheorem{lemma} [thm]{Lemma}
   \newtheorem*{claim}{Claim}

   \newtheorem{cor} [thm]{Corollary}

   \newtheorem*{conjecture*}{Conjecture}

{\theoremstyle{definition}
 
          \newtheorem*{exercise*}{Exercise}
     \newtheorem{example}[subsubsection]{Example}

   \newtheorem{defi}[thm] {Definition}
     \newtheorem{remark} [thm]{Remark}

\newcommand{\RR}{{\mathbb{R}}}

\newcommand{\C}{{\mathbb{C}}}
\newcommand{\QQ}{{\mathbb{Q}}}

\newcommand{\Z}{{\mathbb{Z}}}

 \newcommand{\cA}{{\mathcal A}}

\newcommand{\cC}{{\mathcal C}}

\newcommand{\cJ}{{\mathcal J}}

\renewcommand{\cL}{{\mathcal L}}
\newcommand{\cM}{{\mathcal M}}

\newcommand{\cO}{{\mathcal O}}
\newcommand{\cP}{{\mathcal P}}

\newcommand{\cS}{{\mathcal S}}

 \newcommand{\md}{{\underline{d}}}

\def\<{\langle}
\def\>{\rangle}

\newcommand{\Spec}{\operatorname{Spec}}

\newcommand{\Sing}{{\operatorname{Sing}}}

\newcommand{\Jac}{{\operatorname{Jac }}}
\newcommand{\Pic}{{\operatorname{{Pic}}}}

\newcommand{\Div}{{\operatorname{Div}}}
\newcommand{\dv}{{\operatorname{div}}}
\newcommand{\Prin}{{\operatorname{Prin}}}

\newcommand{\Ker}{{\operatorname{Ker}}}
\newcommand{\Aut}{{\operatorname{Aut}}}

\newcommand{\ocM}{\overline{{\mathcal M}}}

\newcommand{\red}{{\operatorname{red}}}

\newcommand{\trop}{{\operatorname{trop}}}

\newcommand{\an}{{\operatorname{an}}}
\newcommand{\val}{{\operatorname{val}}}

\newcommand{\double}{\genfrac..{0pt}1
{\raise -2pt\hbox{$\scriptstyle\longrightarrow$}}{\raise 4pt\hbox
{$\scriptstyle\longrightarrow$}}}

 \newcommand{\la}{\longrightarrow}
\newcommand{\ha}{\hookrightarrow}

\newcommand{\ov}{\overline}

\newcommand{\Alb}{\operatorname{Alb}}

 \newcommand{\Mgbst}{\ov{\mathcal{M}_g}}
  \newcommand{\Mgban}{\Mgb^{\operatorname{an}}}
 
\def\Mgb{\overline{M}_g}

\newcommand{\PX}{\overline{P_{g-1}(X)}}

\newcommand{\SG}{{{\rm S}_{g}}}

\newcommand{\Mt}{M^{\rm{{trop}}}}

\newcommand{\Mgt}{{M_{g}^{\rm trop}}}

\newcommand{\Agt}{A_g^{\rm trop}}

\newcommand{\Mgtb}{\ov{\Mgt}}

\newcommand{\ST}{\mathcal{ST}}
\newcommand{\SP}{\mathcal{SP}}
 
\begin{document}

\title{ Tropical   methods 
 in the moduli theory of algebraic  curves}

\author[Caporaso]{Lucia Caporaso}

 \address[Caporaso]{Dipartimento di Matematica e Fisica\\
 Universit\`{a} Roma Tre\\Largo San Leonardo Murialdo \\I-00146 Roma\\  Italy }\email{caporaso@mat.uniroma3.it}

\date{\today}




\maketitle

\tableofcontents
  
 \section{Introduction and Notation}

 In recent years a series of remarkable advances in tropical geometry and in non-Archimedean  geometry have brought new insights to the moduli theory of algebraic curves and their Jacobians.  
  The goal of this expository paper is to present some of the results in this area.

There are   some important aspects about the interplay between the theories of tropical and algebraic curves   which we do not discuss here,  to keep the paper to a reasonable length. In particular, we leave out Brill-Noether theory, which has  already seen a fruitful interchange between the two areas over the last years  and is currently progressing at a fast pace.

  We begin the paper by defining, in Section~\ref{TropicalSc}, tropical curves as weighted metric graphs, we then construct their moduli space together with its compactification, the moduli space of extended tropical curves.
  
In Section~\ref{FromSc}, after introducing Deligne-Mumford stable curves   and their moduli space, we   explain the connection between tropical curves and families of stable curves parametrized by a local scheme. We then show how to globalize this connection by introducing the Berkovich analytification of the moduli space of stable curves and showing that it has a canonical retraction onto the moduli space of extended tropical curves, which is in turn identified with the skeleton of the moduli stack   of stable curves.

In Section~\ref{CurvesSc},  highlighting   combinatorial aspects, we describe Jacobians of nodal curves and    their models over discrete valuation rings, focusing on   the N\'eron model and the Picard scheme.
We then turn to compactified Jacobians, introduce a compactification of the moduli space of principally polarized abelian varieties 
and  the compactified Torelli map.

We devote Section~\ref{TorelliSc} to the Torelli theorem for graphs, tropical curves, and stable curves. Although the proofs of these   theorems are logically independent, they  use  many of the same combinatorial ideas, which we tried to highlight in our exposition.

The order  we chose to present the various topics does not reflect the actual chronology. 
As we said, we begin with tropical curves and their moduli space,
then discuss the connection with algebraic stable curves also using    analytic spaces;
we treat the Torelli problem   at the end. History goes almost in the opposite direction:
moduli spaces for tropical curves and abelian varieties were  rigorously constructed after 
  the first Torelli theorem was proved, and in fact the need for such moduli spaces
was explicitly brought up by the   Torelli theorem; see the appendix in \cite{CV1}.
The connection with   algebraic curves via non-Archimedian geometry was established a few years later.
Although our presentation is not consistent with    history, it is, in our opinion,
a   natural path through the theory,
starting from the most basic objects (the curves) and ending with the most  complex
(the Jacobians).

 In writing this paper we tried to address the non-expert reader, hence we included various well known statements in the attempt
 of making the material  more accessible; of course, we claim no originality for that.

   \subsection{Notation}
  Unless we  specify otherwise, we shall use the following notations and conventions.

  We work over an algebraically closed field $k$.
  
We denote by $K$   a field containing $k$  and assumed   to be complete with respect to a non-Archimedean valuation, written
$$
v_K:K \to \RR\cup \{\infty\}.
$$ 
Such a $K$ is also called, as in \cite{berkovich},   a {\it non-Archimedean field}.
We assume  $v_K$   induces   on $k$ the trivial valuation, $k^*\to 0$.
We denote by    $R$ the valuation ring of $K$.
  
In this paper any field extension  $K'|K$  is an extension of non-Archimedean fields so that  $K'$ will be endowed with a valuation
 inducing $v_K$.

 Curves are always assumed to be  reduced, projective and having at most nodes as singularities.
   
     $g\geq 2$ is an integer.
  
  $G=(V,E)$ a finite graph (loops and multiple edges allowed), $V$ and $E$ the sets of vertices and edges.
  We also write $V=V(G)$ and $E=E(G)$.
  Our graphs will be connected.
    
    ${\bf{G}}=(G,w)$ denotes a weighted graph.
  
  $\Gamma=(G,w,\ell)$ denotes a   tropical curve, possibly extended.

We use the    ``Kronecker" symbol: for two   objects $x$ and $y$
 $$
 \kappa_{x,y}:=\begin{cases}
1 & \text{ if } x=y, \\
0 & \text{ otherwise.}
\end{cases}
 $$

In drawing graphs   we shall denote by a ``$\circ$" the vertices of weight $0$ and by a ``$\bullet$" the vertices of positive weight.

\

\noindent
{\it Acknowledgements.} The paper benefitted from comments and suggestions   from  Sam Payne,   Filippo Viviani, and    the   referees, to whom  I am   grateful.

 \section{Tropical curves}
 \label{TropicalSc}
 \subsection{Abstract tropical curves}
 We begin by defining abstract tropical curves, also known as ``metric graphs", following, with a slightly different terminology,  \cite{MIKnotices}, \cite{MZ}, and \cite{BMV}.

\begin{defi}
 A {\it  pure  tropical curve} is a pair $\Gamma = (G, \ell)$ such that $G=(V,E)$ is a  graph
and $\ell:E\to \RR_{>0}$ is a {\it length} function on the edges.
 
More generally, a  {\it  (weighted) tropical curve} is a triple $\Gamma = (G,    \ell, w)$ such that $G$ and $\ell$ are as above,
and $w:V\to \Z_{\geq 0}$ is a {\it weight} function on the vertices.

     \end{defi}
     
Pure tropical curves shall be viewed as   tropical curves whose weight function is identically $0$. 

The {\it genus} of the tropical curve $\Gamma = (G,    \ell, w)$ is defined as the genus of the underlying weighted graph $(G,w)$, that is:
 
\begin{equation}
\label{genus}
  g(\Gamma):=g(G,w):=b_1(G)+\sum_{v\in V}w(v),
  \end{equation} 
where $b_1(G)=rk_{\Z} H_1(G,\Z)$ is, as usual, the first betti number of $G$.

Pure tropical curves arise naturally in tropical geometry, and they were the first ones to be studied for   sometime;
the need to generalize the definition by adding weights on the vertices appeared  when studying families.
These families are quite easy to generate by varying the lengths of the edges; 
 they have a    basic invariant in    the genus   defined above, which, for pure tropical curves, equals the first betti number.

Now, the problem   one encounters when dealing with such families is that if some edge-length tends to zero the genus of the limiting graph may decrease, as in the following example.
 
\begin{example}
 \label{puredeg}
 Let us pick a pure tropical curve of genus $2$, drawn as the first graph from the left in the picture below.
 Now consider a degeneration in  three steps  by letting the edge lengths, $l_1, l_2, l_3$,  go to zero one at the time while the other lengths remain fixed;
the picture represents the three steps. It is clear that after the second and third step the genus decreases each time,
and the last graph has genus $0$.
 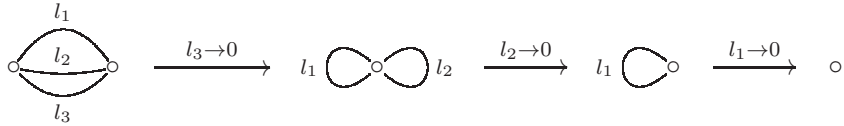
\begin{figure}[h]
\begin{equation*}
\xymatrix@=.5pc{
&&&&&&&&&&&&&&\\
*{\circ}\ar@{-} @/^ 1.2pc/[rrr]^{l_1}\ar@{-} @/_.2pc/[rrr]^{l_2}\ar@{-} @/_.9pc/[rrr]_{l_3}  &&&*{\circ}  
  &\ar@{->}[rrrr] ^{l_3 \to 0} &&&&&&&*{\circ}\ar@{-}@(ur,dr)^{l_2}\ar@{-}@(ul,dl)_{l_1} 
  &&&\ar@{->}[rrr]^{l_2 \to0} &&&&&&*{\circ}\ar@{-}@(ul,dl)_{l_1} 
 &\ar@{->}[rrr]^{l_1 \to 0} &&&&*{\circ}\\
 &&&&&&&&&}
\end{equation*}
\caption{Specialization of pure tropical curves}
\end{figure}
\end{example}
 
A remedy to this ``genus-dropping" problem, introduced in \cite{BMV},
  is to add  a piece of structure, consisting of a weight function on the vertices, 
and define {\it weighted  (edge) contractions} in such a way that as a 
loop-length goes to zero  the weight of its   base vertex
increases by $1$. Moreover,
when the length of a non-loop edge, $e$, tends to zero, $e$  gets contracted to a vertex whose weight is defined as the sum of the weights of the ends of $e$. 

For example, let us modify the above picture by assuming   the initial curve has both vertices of weight $1$, so its genus is $4$. As $l_3$ tends to zero   the limit curve has a vertex of  weight $1+1=2$. In the remaining two steps we have a loop getting contracted,
so the last curve is an isolated vertex of weight $4$.   

\begin{figure}[h]
\begin{equation*}
\xymatrix@=.5pc{
&&&&&&&&&&&&&&\\
*{\bullet}\ar@{-} @/^ 1.2pc/[rrr]^{l_1}
\ar@{-} @/_.2pc/[rrr]^{l_2}\ar@{-} @/_.9pc/[rrr]_{l_3}_>{\scriptstyle {1}}_<{\scriptstyle {1}}  &&&*{\bullet}  
  &\ar@{->}[rrrr] ^{l_3 \to 0} &&&&&&&*{\bullet}\ar@{-}@(ur,dr)^{l_2}\ar@{-}@(ul,dl)_{l_1}_(.95){\scriptstyle {2}} 
  &&&\ar@{->}[rrr]^{l_2 \to0} &&&&&&*{\bullet} \ar@{-}@(ul,dl)_{l_1}_(1){\scriptstyle {3}}
 &\ar@{->}[rrr]^{l_1 \to 0}&&&&*{\bullet  _{\scriptstyle {\ 4}}} &  \\
 &&&&&&&&&&&&&&&&&&&&&&&&&&}
\end{equation*}
\caption{Specialization of (weighted) tropical curves}
\end{figure}
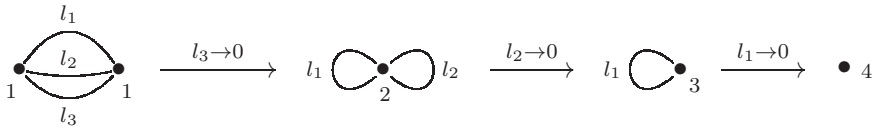

It is quite clear that weighted (edge) contractions can be defined  combinatorially, disregarding the length function.
Hence 
we shall denote by 
\begin{equation}
 \label{wct}
(G,w) \to (G',w')
\end{equation}
a weighted contraction  as defined above.  Notice that  there is a natural inclusion $E(G')\subset E(G)$, and a surjection $V(G)\to V(G')$.
As we said,  weighted contractions leave  the genus invariant. 
 
 \
 
One thinks of a vertex $v$ of positive weight $w(v)$ as having $w(v)$  invisible loops of zero length based at it;  this
 somewhat naive interpretation turns out to  work well in developing a Riemann-Roch divisor theory for (weighted) tropical curves, as carried out in \cite{AC}.

From the point of view of algebraic geometry, the notion of weighted tropical curve is quite natural for at least two reasons.
First of all, it resembles the notion of non-reduced schemes (some subsets  of which have a   ``multiple" structure); 
moreover, just as in the tropical set-up where weights are introduced to better understand families and their limits, non reduced-schemes in algebraic geometry are indispensable to study families of algebraic varieties.
Secondly, weighted graphs (with no length function) have been used since the time of \cite{DM} to describe the moduli space of stable algebraic curves, as we shall explain in Section~\ref{FromSc}. 

According to these analogies, a vertex $v$ of weight $w(v)$ corresponds to a component of geometric genus $w(v)$  (the geometric genus of an irreducible curve is the genus of its desingularization). For example, in the last step of Figure 2, the last graph corresponds to a smooth curve of genus $4$, and  the previous graph corresponds to a   curve of geometric genus $3$ with one node.
In this    analogy, the arrow in the picture has to be reversed  and interpreted as a family of smooth curves of genus $4$ specializing to a curve (of geometric genus $3$) with one node.

  \subsection{Equivalence of tropical curves}
  
Two tropical curves, $\Gamma = (G,    \ell, w)$ and $\Gamma ' = (G',    \ell', w')$ are {\it isomorphic} 
if there is an isomorphism between $G$ and $G'$ which preserves both the weights of the vertices
and the lengths of the edges.

Isomorphism is thus  a natural notion; on the other hand
 tropical curves are defined to be  {\it equivalent} by a more general relation.

\begin{defi}
 \label{tropeq}

Two tropical curves,  $\Gamma$ and $\Gamma '$,  of genus at least $2$ are {\it equivalent} if   one obtains   isomorphic tropical curves, $\ov{\Gamma}$ and $\ov{\Gamma'}$,
 after performing   the following two operations   until $\ov{\Gamma}$ and $\ov{\Gamma'}$ have  no
  weight-zero vertex of valency
 less than $3$.
 
- Remove all  weight-zero vertices of valency
 $1$  and their adjacent edge.
 
- Remove every weight-zero vertex $v$ of valency $2$   and replace it by a point,
so that the two edges adjacent to $v$ are transformed into one edge.
 \end{defi}
 
It is clear that the underlying graph, $(\overline{G}, \overline{w})$, of  $\ov{\Gamma}$ has the property that every vertex of weight zero has valency at least $3$ (the same holds for  $\ov{\Gamma'}$, obviously).
Weighted graphs of genus at least $2$ with this property are called  {\it stable}; they are old friends of algebraic geometers as we shall see in Section~\ref{FromSc}. One  usually refers to $\ov{\Gamma}$ as the {\it stabilization} of $\Gamma$.

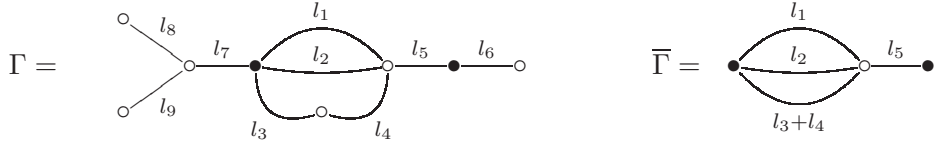
\begin{figure}[h]
\begin{equation*}
\xymatrix@=.5pc{
 &&*{\circ}&&&&&&&&&&&\\
  \Gamma =&&&&*{\circ}\ar@{-}[rr]^{l_7}\ar@{-}[dll]^{l_9}\ar@{-}[ull]_{l_8}&&*{\bullet}\ar@{-} @/^ 1.2pc/[rrrr]^{l_1}
\ar@{-} @/_.2pc/[rrrr]^{l_2}\ar@{-} @/_.9pc/[rrd]_{l_3} &&&&*{\circ}\ar@{-} @/^.9pc/[lld]^{l_4} \ar@{-}[rr]^{l_5}&&*{\bullet}\ar@{-}[rr]^{l_6}&&*{\circ}
& &&&  \overline{\Gamma} =&
*{\bullet}\ar@{-} @/^ 1.2pc/[rrrr]^{l_1}
\ar@{-} @/_.2pc/[rrrr]^{l_2}\ar@{-} @/_1.2pc/[rrrr]_{l_3+l_4}   &&&&*{\circ} \ar@{-}[rr]^{l_5}&&*{\bullet} &&  \\
&&*{\circ}&&&&&&*{\circ}&&&&&&&&&&&&&}
\end{equation*}
\caption{A tropical curve $\Gamma$ and its stabilization, $\ov{\Gamma}$.}
\end{figure}

Observe that while for any fixed genus there exist infinitely many (non-isomorphic) graphs,   only finitely many of them are stable.

The following instructive simple  lemma will also  be useful later.

\begin{lemma}
\label{3g-3}
 Let $(G,w)$ be a stable graph of genus $g\geq 2$.
 Then $G$ has at most $3g-3$ edges, and the following are equivalent.
 
\begin{enumerate}
\item
\label{G1}
 $|E(G)|=3g-3.$
  \item
  \label{G2}
 Every vertex of $G$ has weight $0$ and valency $3$.
 \item
   \label{G3}
 Every vertex of $G$ has weight $0$ and $|V(G)|=2g-2$.
 \end{enumerate}
\end{lemma}
\begin{proof}
For any fixed genus   there exist only finitely many stable graphs, hence there exists a maximum, $M$, on the number of edges
that a stable graph of that genus can have.
Let $G$ be a graph having $M$ edges, then no vertex of $G$ has positive weight, for otherwise we could replace this vertex
by a weight-zero vertex with as many loops as its weight attached to it; this would increase  the number of edges without changing the genus or loosing stability.

So, every vertex of $G$ has weight $0$, hence, by stability, it has valency at least $3$, and we have
$3|V|\leq 2M$.
We obtain
$$
g=M-|V|+1+\sum_{v\in V}w(v)\geq M-2M/3+1 = M/3+1
$$
hence $M\leq 3g-3$, as claimed. 
By the above inequality,
 if $M=3g-3$     every vertex has (weight $0$ and) valency equal to $3$, proving \eqref{G1} $\Rightarrow $ \eqref{G2}.

For \eqref{G2} $\Rightarrow $ \eqref{G3}, we have $3|V|/2= |E| $
hence
$$
g= 3|V|/2-|V|+1 = |V|/2+1
$$
as wanted.
For \eqref{G3} $\Rightarrow $ \eqref{G1}, we have
 $
g= |E|-2g+2+1 = |E|-2g+3,
$ 
and we are done.
\end{proof}

Now,    the set of equivalence classes of   tropical curves of   genus $g\geq 2$ 
forms a nice moduli space, denoted here by $\Mgt$,  
constructed  in \cite{BMV} to which we refer for details (see also \cite {MIK5}).
Let us, for the moment, content ourselves to recall that  $\Mgt$ has a natural structure of topological space; it is   a ``stacky fan" in the terminology of  \cite{BMV}.  It is partitioned as follows
\begin{equation}
 \label{Mgt}
\Mgt = \bigsqcup _{(G,w) \in \SG} \Mt (G,w) 
\end{equation}
where $\SG$ denotes the   set of all  stable  weighted graphs of genus $g$   and
$ \Mt{(G,w)}$ is the set of all isomorphism classes of tropical curves  whose underlying   graph is $(G,w)$.

\begin{remark}
 The    description of $\Mgt$ given in \eqref{Mgt} implies that for every equivalence class of tropical curves there is a unique representative, up to isomorphism, whose underlying   graph is stable.  
 Therefore from now on we shall usually assume 
that our tropical curves have   stable underlying   graph. Such tropical curves are also called ``stable". So,  two stable tropical curves are isomorphic if and only if they are equivalent.
\end{remark}

  \subsection{Constructing the moduli space of tropical curves}
We shall here describe how to give a topologically meaningful structure to $\Mgt$.

Let us start from the stratification given in \eqref{Mgt}.
It is not hard to describe   each stratum $\Mt{(G,w)}$. Write $G=(V,E)$ and consider the open cone $\RR^{|E|}_{>0}$  with its euclidean topology. To every point $(l_1,\ldots, l_{|E|})$ in this cone there corresponds a unique tropical curve whose 
$i$-th edge has length $l_i$. Now consider the   automorphism group, $\Aut(G,w)$, of $(G,w)$; 
by definition, $\Aut(G,w)$ is the  set of automorphisms of $G$ which preserve the weights on the vertices; in particular 
 we have a homomorphism
from 
 $\Aut(G,w)$ to the symmetric group on $|E|$ elements.
Hence $\Aut(G,w)$
acts on  $\RR^{|E|}_{>0}$  by permuting the coordinates, and  the quotient by this action, endowed with the quotient topology,  is the space of isomorphism classes of tropical curves having $(G,w)$ as underlying graph:
$$
 \Mt{(G,w)} =\RR^{|E|}_{>0}/\Aut(G,w).
$$

Now,   the boundary of the closed cone $\RR^{|E|}_{\geq 0}$ naturally parametrizes tropical curves  with fewer edges   that  are     specializations, i.e.  weighted contractions, of tropical curves in the open cone. 
Indeed, the closure in $\Mgt$ of a stratum as above is a union of strata, and we have
\begin{equation}
 \label{poset}
 \Mt{(G,w)} \subset \overline{ \Mt {(G',w')}} \quad \Leftrightarrow \quad  (G',w') \to (G,w)
\end{equation}
(notation in \eqref{wct}) which gives an interesting  partial ordering on the set of strata appearing in \eqref{Mgt}.

It turns out    that the action of $\Aut(G,w)$ extends on the closed cone in such a way that the quotient
$$
 {\widetilde \Mt {(G,w)}} =\RR^{|E|}_{\geq 0}/\Aut(G,w) 
$$
identifies isomorphic curves, and hence   maps to $\Mgt$.

Now, there are a few equivalent ways to  construct $\Mgt$ as a topological space;  we proceed starting from  its ``biggest" strata (i.e. strata of maximum dimension), as follows. By Lemma~\ref{3g-3},  the maximum number of edges of a stable graph of genus $g$ is $3g-3$. Such graphs  correspond to our biggest strata, and we proceed by
  considering the map 
$$
 \bigsqcup _{{(G,w) \in \SG:}\atop{ |E|=3g-3}} {\widetilde \Mt{(G,w)}} \la \Mgt
$$
mapping a curve to its isomorphism  
class, as we observed above. Now, by Proposition~\ref{toptrop}  below,
  this map is surjective,
i.e.    every stable tropical curve can be obtained  as a specialization of a stable tropical curve with $3g-3$ edges.
Therefore we can endow $\Mgt$ with the quotient topology induced by the space on the left.

\begin{prop}
\label{toptrop}
  Let $(G,w)$ be a stable graph of genus $g$. Then there exists a stable graph $(G',w')$     of genus $g$ having $3g-3$ edges and such that
  $\Mt{(G,w)} \subset \overline{ \Mt {(G',w')}}$.

\end{prop}
\begin{proof}
  
We can  assume   $|E(G)|<3g-3$.
By \eqref{poset},
it suffices to show that there exists a   $(G',w')$ with  $3g-3$ edges  such that
$(G',w') \to (G,w)$.
The  following proof is illustrated on an explicit case in Example~\ref{toptropex}.

  Let  $V_{+}\subset V(G)$ be the set of vertices of positive weight;
  consider the graph $G''$ obtained from $G$ by replacing every $v\in V_{+}$ by a weight-zero vertex with $w(v)$ new loops attached to it.
 
By construction,  $G''$  has all vertices of weight zero, and specializes to $(G,w)$  by contracting every one of the  new loops   ($\sum_{v\in V_{+}}w(v)$ of them).

If every vertex of $G''$ has valency $3$ we are done by Lemma~\ref{3g-3}.

So, suppose
$G''$ has some vertex, $v$, of valency  $N\geq 4$.
We shall construct a graph $G'''$ which specializes to $G''$ by contracting to $v$  
one edge whose ends have both valency less than $N$. Iterating this construction until  there are no vertices of valency more than $3$ we are done.

We   partition the set, $H_v$, consisting in the $N$   half-edges adjacent to $v$,  into two subsets,
$H_v^1$ and $H_v^2$, of respective cardinalities $N_1=\lfloor N/2\rfloor$ and $N_2=\lceil N/2\rceil$.
As $n\geq 4$ we have
$$
2\leq N_i\leq N-2.
$$
Consider the graph $G'''$ obtained from $G''$ by  replacing $v$ by a non-loop edge $e$ whose ends, $u_1$ and $u_2$,
are attached to, respectively, $H_v^1$ and $H_v^2$. As $u_i$ has valency $N_i+1$, the graph $G'''$ is stable and
its vertices $u_i$ have both valency less than $N$.
It is clear that contracting $e$ in $G'''$ gives back our $G''$. So we are done.
\end{proof}
\begin{example}
\label{toptropex}
 The following picture illustrates the proof of Proposition~\ref{toptrop} on the genus-$2$ graph $G$ consisting of one vertex with weight $2$ and no edges. 
 
 On the right   we see the two possible  
 graphs  $G'$,  corresponding (in the proof) to   different distributions in $G''$ of the four half-edges adjacent to  $v$.
 Of course, $G''$ is obtained from $G'$ contracting the edge $e$, and $G$ is obtained by contracting both loops of $G''$.

 \begin{figure}[h]
\begin{equation*}
\xymatrix@=.5pc{
 &G&&&&&&&&&&G''&&&&&&&&&&G'\\
 &&&&&&&&&&&&&&\\
 &{\bullet   _{\scriptstyle {\ 2}}}&&&&&&\ar@{->}[llll]   &&&&*{\circ}\ar@{-}@(ur,dr) ^(.06){v}\ar@{-}@(ul,dl)  
 &&&&&&&\ar@{->}[llll] &&*{\circ}\ar@{-} @/^ 1pc/[rr] \ar@{-} @/_.2pc/[rr] \ar@{-} @/_.9pc/[rr]_{e}_(.01){u_1}_(.9){u_2}  &&*{\circ} &&&&&&&\\
    &&&&&&&&&&&& &&    &&&&&&&&&&&& &&\\    
    &&&&&&&&&&&& &&    &&&&&&&&&&&& &&\\
&&&&&&&&&&&&&&&&&&\ar@{->}[lllluu] &&*{\circ}\ar@{-}@(ul,dl)  \ar@{-}[rr]_{e}^(.01){u_1}^(.9){u_2}&&*{\circ}\ar@{-}@(ur,dr) &&
}
\end{equation*}
\caption{Proof of Proposition~\ref{toptrop}}
\end{figure}
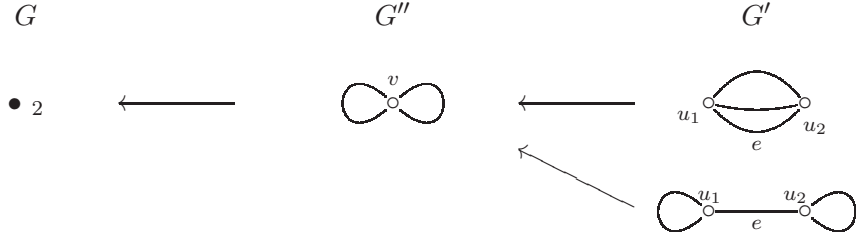
\end{example}
As it turns out,  the topological space $\Mgt$ 
is connected, 
 Hausdorff, and of pure dimension $3g-3$ (i.e. it has a dense open subset which is a $(3g-3)$-dimensional orbifold over $\RR$).
 More details can be found in \cite{BMV} and \cite{Chbk}; see also \cite{CGP} for some recent results on  the topology   of $\Mgt$.

\subsection{Extended tropical curves} Let us now focus on the fact that, as it is easy to check, $\Mgt$ is not compact; this 
is caused by
  edge-lengths being allowed to grow arbitrarily. In fact, we explained before how adding weights on the vertices enabled us to control edge-lengths going to zero; on the other hand  we still cannot control edge-lengths going to infinity.

This problem can be solved    by further generalizing the definition of tropical curve, by allowing edges of infinite length   (this is done in  \cite{Chbk} using an idea of G. Mikhalkin).

\begin{defi}
  An {\it extended tropical curve} is a triple $\Gamma = (G,    \ell, w)$ where $(G,w)$ is a stable graph
  and $\ell:E \to \RR_{>0}\cup \{\infty\}$ an ``extended" length function.
 \end{defi}

Now,  we
compactify    $\RR \cup \{\infty\}$ by the Alexandroff one-point compactification, and consider its subspaces with the induced topology.
The moduli space of extended tropical curves with fixed underlying weighted graph is
$$ 
 \overline{\Mt  {(G,w)}}={{( \RR_{>0}\cup \{\infty\})^{|E|}}\over{\Aut(G,w)}} 
$$
  with the quotient topology. Now, just  as we did for $\Mgt$, we  extend  the action of $\Aut(G,w)$ on the closure
  of $( \RR_{>0}\cup \{\infty\})^{|E|}$, so that the quotient  $$
 {\widetilde {M_{\infty}^{\trop}}}(G,w)=  {{( \RR_{\geq 0}\cup \{\infty\})^{|E|}}\over{\Aut(G,w)}} 
$$
is a compact topological space which maps  to the moduli space    for extended tropical curves, written $\Mgtb$.
Arguing as in the previous subsection,  
we have a surjection
 \begin{equation}
 \label{Mgtb}
\bigsqcup _{{(G,w) \in \SG:}\atop{ |E|=3g-3}}  {\widetilde {M_{\infty}^{\trop}}}(G,w) \la   \Mgtb = \bigsqcup _{(G,w)\in \SG} \overline{\Mt  {(G,w)}}.
\end{equation}
It is not hard to prove that $\Mgtb$, with the quotient topology, is    compact, normal, and contains $\Mgt$ as dense open subset.

\begin{remark}
  As we shall see in the next section, the introduction of extended tropical curves 
 has a precise meaning in relating tropical curves to algebraic curves. 
 A tropical curve will be shown to correspond  to   families of smooth curves degenerating to nodal ones,
while an extended tropical curve  will correspond  to   families of nodal curves  degenerating,    again, to nodal ones.
For instance, under this correspondence  an extended tropical curve $\Gamma = (G, w, \underline{\infty})$, all of whose edges have length equal to $\infty$, will correspond  to locally trivial families all of whose fibers have dual graph $(G,w)$; see Proposition~\ref{corrK}.  
 \end{remark}
 \
 
  \section{From algebraic curves to tropical curves}
  \label{FromSc}
  The primary goal of this section is to show that, as we just mentioned,  
   tropical curves are associated to degenerations of smooth curves  to   singular ones.
     We first need to introduce   some terminology and some conventions.

  \subsection{Algebraic curves}
    
 By {\it algebraic curve} we mean a     projective  variety of  dimension one over an algebraically closed field,
unless we specify otherwise. For reasons that will be explained in the next subsection, we shall be interested exclusively   in {\it nodal} curves, i.e. reduced (possibly reducible) curves admitting at most nodes as singularities, and we shall usually omit to mention it. 

To a   curve $X$   one associates its {\it (weighted) dual graph}, written $(G_X, w_X)$ and defined as follows.  The vertex-set of $G_X$  is the set of irreducible components of $X$, and the weight of a vertex/component
 is its geometric genus (i.e. the genus of its desingularization);  the edge-set
  is the set of nodes of $X$, with the ends of an edge/node equal to the
irreducible components of $X$ on which the   node lies, so that 
 loops   correspond to nodes of irreducible components.
 
 In symbols, for any vertex $v$ of  $G_X$ we write $C_v\subset X$ for the corresponding irreducible component;  $w_X(v)$ is the genus of the normalization, $C_v^{\nu}$, of $C_v$. We shall abuse notation by using the same symbol
for a node of $X$ and the corresponding edge of $G_X$.

From now on, we shall assume that $X$ is a  connected  curve.

We define $X$ to be {\it stable} if so is its dual graph  $(G_X, w_X)$.

\begin{remark}
 It is well known that  a connected   curve  is stable if and only if it has finitely many  automorphisms,
 if and only if its dualizing line bundle is ample; see \cite{DM}, \cite{HM}  or \cite {gac}.
 \end{remark}

We want to highlight the connection between a curve and its dual graph. Let us begin 
by proving the following basic
\begin{claim}
A  curve $X$ has the same (arithmetic) genus as its dual graph.
  \end{claim}
First of all, recall that the   genus of $X$ is defined as 
 $g=h^1(X,\cO_X)$.
  
Now, write $G_X=(V,E)$, and consider    the normalization map
$$
\nu:X^{\nu}=\bigsqcup_{v\in V}C_v^{\nu}\la X.
$$ 
The associated map of   structure sheaves yields an exact sequence
\begin{equation}
\label{smap}
0\la \cO_X\la \nu_* \cO_{X^{\nu}} \la \cS\la 0
\end{equation}
where $\cS$ is a skyscraper sheaf supported on the nodes of $X$; 
 the associated exact sequence in cohomology is as follows (identifying the cohomology groups of  $\nu_* \cO_{X^{\nu}} $
 with those of $\cO_{X^{\nu}}$ as usual)
 
\begin{eqnarray*}
 0\la & H^0(X,\cO_X)\la & H^0(X^{\nu},\cO_{X^{\nu}})\stackrel{\tilde{\delta}}{\la} k^{|E|}\la \\
\la & H^1(X,\cO_X)\la & H^1(X^{\nu},\cO_{X^{\nu}})\la  0.\\
 \end{eqnarray*}
Hence
$$
g=h^1(X^{\nu},\cO_{X^{\nu}})+|E| -|V| +1=\sum_{v\in V}g_v+b_1(G_X)=g(G_X,w_X) 
$$
where $g_v=h^1(C_v^{\nu},\cO_{C_v^{\nu}})$ is the  genus of $C_v^{\nu}$.  By \eqref{genus}  we conclude that $X$ and $(G_X,w_X) $  have the same genus. The claim is proved.
 
\subsection{A graph-theoretic perspective on the normalization of a curve}
 \label{gtp}
We shall continue the preceding  analysis and   prove  that the map $\tilde{\delta}$ in the   cohomology sequence above can be identified with the coboundary map of the graph $G_X$.

Let us  first describe $\tilde{\delta}$. Let $ {\Phi}= \{\phi_v\}_{v\in V}$ be an element in $H^0(X^{\nu},\cO _{X^{\nu}}) $,
so that $\phi_v\in H^0(C_v^{\nu}, \cO _{C_v^{\nu}})\cong k $ for every vertex $v$.  Now  $\tilde{\delta}$ is  determined up
to a choice of sign for each node of $X$, i.e. for each edge in $E$; this choice of sign amounts to the following. Let $e\in E$ be a node of $X$,
then $e$ corresponds to two (different) branch-points on the normalization, $X^{\nu}$, of $X$, each of which lies in a component of $X^{\nu}$;
let us denote by $v_e^+$ and $v_e^-$ the vertices corresponding to these two (possibly equal) components of $X^{\nu}$
(so that the node $e$ lies in $C_{v_e^+}\cap C_{v_e^+}$).
Now we have
$$
\tilde{\delta}(\Phi) =  \{\phi_{v^+_e}- \phi_{v^-_e}\}_{e\in E}.
$$
To describe this map in graph-theoretic terms
we need a brief excursion into   graph theory, which will also be useful later in the paper.
 
For a pair of   vertices, $v,w$, we set
\begin{equation}
 \label{[vw]}
 [v,w]:=\kappa_{v,w}=\begin{cases}
1 & \text{ if } v=w, \\
0 & \text{ otherwise;}
\end{cases}
\end{equation}
this gives  a bilinear pairing on $C_0(G_X, k)$,   the $k$-vector space having $V$ as basis.

An
  orientation on $G_X$ is a pair
of maps $s, t:E\to V$ so that  every edge $e$ is oriented from $s(e)$ to $t(e)$,
  where  $s(e)$ and $t(e)$ are the ends of $e$. Now we can define  the boundary map
 
  $$
  C_1(G_X, k )\stackrel{\partial }{\la}   C_0(G_X, k ); \quad \quad \sum_{e\in E}c_e e\mapsto \sum_{e\in E}c_e \bigr(t(e)-s(e)\bigl) 
  $$
   (where $C_1(G_X, k)$ is the $k$-vector space having $E$ as basis).
We need to define the coboundary map,
  $$
  C_0(G_X, k)\stackrel{\delta}{\la} C_1(G_X, k). 
$$
We define  $\delta$   as the   linear extension of the following: for every $v\in V$ and $e\in E$ we set, using \eqref{[vw]},
 $$
 \delta (v):=\sum_{e\in E}\bigr[v, \partial (e)\bigl]e
 $$
 where we have
\begin{equation}
 \label{fve}
\bigr[v, \partial (e)\bigl] = \begin{cases}
1 & \text{ if } v=t(e) \text{ and } t(e)\neq s(e_), \\
-1 & \text{ if } v=s(e) \text{ and } t(e)\neq s(e_),\\
0 & \text{ otherwise.}
\end{cases}
\end{equation}

Now, with the notation at the beginning of this subsection,
we assume that   the orientation maps $s$ and $t$ are such that $s(e)=v_e^-$ and $t(e)=v_e^+$.
We have 
$$
 H^0(X^{\nu},\cO_{X^{\nu}})\stackrel{\alpha}{\la} C_0(G_X, k)\stackrel{\delta}{\la} C_1(G_X, k)\stackrel{\beta}{\la}   k ^{|E|} 
$$
where $\alpha$ and $\beta$ are the obvious isomorphisms. In particular,   any vertex $u\in V$ viewed as element of $C_0(G_X, k)$  is  the image by $\alpha$ of 
$\Phi^u=\{\phi_v^u\}_{v\in V}$ defined as follows
$$
\phi_v^u: =\kappa_{u,v}.
$$ 
By definition, we have
$ 
\tilde{\delta}(\Phi^u) =  \{\phi^u_{v^+_e}- \phi^u_{v^-_e}\}_{e\in E},
$ where
$$
\phi^u_{v^+_e}- \phi^u_{v^-_e}= \begin{cases}
1 & \text{ if } u=v^+_e \text{ and } v^+_e\neq v^-_e, \\
-1 & \text{ if } u=v^-_e \text{ and } v^+_e\neq v^-_e,\\
0 & \text{ otherwise.}
\end{cases}
$$
Comparing with \eqref{fve} we  obtain
 $$
\tilde{\delta}(\Phi^u)  = \beta\circ \delta (u) = \beta\circ \delta\circ \alpha  (\Phi^u).
 $$
Since  $\{ \Phi^u, \  \forall u\in V$\} is a basis for $ H^0(X^{\nu},\cO_{X^{\nu}})$ we conclude that
 the map $\tilde{\delta}$  is naturally identified with the coboundary map $\delta$
 of the graph $G_X$; more precisely we proved the following.

\begin{prop} Notations as above. Then
$ \tilde{\delta} = \beta\circ \delta\circ \alpha $ and we have  the following  exact sequence
$$
0\la H_1(G_X,k)\la H^1(X,\cO_X)\la H^1(X^{\nu},\cO_{X^{\nu}})\la 0.
$$
\end{prop}
  \subsection{Families of algebraic curves over local schemes}
 As we mentioned,
from the point of view of tropical geometry, smooth algebraic curves are particularly interesting when considered in families
 specializing to singular curves. 
We shall need a local analysis,
hence we shall concentrate on families parametrized by a valuation scheme  or,   even more ``locally",
by a complete valuation scheme.
  
Now, it is well known that  in families of smooth curves
 singular degenerations are, in general, unavoidable. This phenomenon is reflected in the fact that $\Mgb$, the moduli scheme  of smooth curves of genus $g$,  is not projective.
 
 On the other hand for many purposes, including ours, it suffices to consider only degenerations to nodal curves, ruling out all other types of singularities; in moduli theoretic terms, the moduli space $M_g$ is a dense open subset in an irreducible, normal,  projective scheme, $\Mgb$, which is the moduli space for   stable curves. 
 
 We shall get back to $\Mgb$ later, let us now limit ourselves to the local picture, 
 described by the following well known fact, a variation on the classical Deligne-Mumford Stable Reduction Theorem, in
 \cite{DM}. 
 
 Recall that $K$ denotes a complete valuation field  and   $R$ its  valuation ring.
 
\begin{thm}
\label{srt} Let  $\cC$ be  a stable curve over  $K$.
 
 Then there exists a finite field extension  $ K'|K$   such that the base change $\cC'=\cC \times_{\Spec K} \Spec K'$ admits a unique model over the valuation ring  of $K'$ whose special fiber is a stable curve.
\end{thm}
The theorem is represented  in the following commutative diagram.
 $$\xymatrix@=.5pc{
& \cC'\ar@{->} @/_.6pc/[rrrdd]\ar[dd]\ar@{^{(}->}[rrr]&&& \cC'_{R'} \ar[rrdd]&&&&&&&&\\
 &&&&&&&&&&&\\
&\cC \ar@{->} @/_.6pc/[rrrdd]  &&&\Spec K'\ar[rr]\ar[dd]&&\Spec R' \ar[dd]\ar[dddrrr]^{\mu_{\cC'_{R'}}}&&\\
&&&&&   && \\
&&  &&\Spec K \ar@{^{(}->}[rr]  \ar@{->} @/_.6pc/[drrrrr]_{\mu_{\cC}}&&\Spec R \ar[drrr]&&&&  \\
 &&&  &&&&&&\Mgb&&
}$$

\begin{proof}
 If $K$ is a discrete valuation field (not necessarily complete) this is the original Stable Reduction Theorem \cite[Cor. 2.7]{DM}.
Our statement is  a consequence of   it, and of some properties of the moduli scheme $\Mgb$.

Indeed, recall that $\Mgb$ is  projective,   it is a coarse moduli space for stable curves, and admits  a finite covering over which there exists a tautological  family of stable curves (i.e. the fiber over a point is isomorphic
to the curve parametrized by the image of that point in $\Mgb$).

Now, since $\Mgb$ is projective, by the valuative criterion of properness  the moduli map $\mu_{\cC}:\Spec K\to \Mgb$ associated to $\cC$ extends to a regular map
  $\Spec R\to \Mgb$, which, however,  needs not come from a family of stable curves over $\Spec R$.
 But, as we observed above, a
finite extension $R \ha R'$ exists such that
the base-changed   curve, $\cC'=\cC\times_{\Spec K} \Spec K'$, over the fraction field, $K'$,  of $R'$, admits a stable model, $\cC'_{R'}$, over $\Spec R'$.  
\end{proof}
  To avoid confusion, let us  point out that the  stable model, $\cC'_{R'}$, does not, in general, coincide with the base change  $\cC_R\times_{\Spec R} \Spec R'$.
  
 Some explicit constructions of   stable  reduction   can be found in \cite[Sect. 3C]{HM}.
 
\begin{remark}
\label{Cs}   
The map $\Spec R\to \Mgb$ in the diagram above is uniquely determined by $\cC$; hence so is
the image of the special point of $\Spec R$. This is   a stable curve defined over  the residue field, $k$,  of $R$ 
which will be denoted by $\cC_k$.
We can thus define a  {\it (stable) reduction} map for our   field $K$:

\begin{equation}
\label{redK}
\red_K: \ocM_g(K) \la \Mgb; \quad \quad \cC \mapsto \cC_k 
\end{equation}
where  $\ocM_g(K)$ denotes the set of   stable  curves of genus $g$ over $K$.

Let now $\Mgb(K)$ denote, as usual, the set of $K$-points of $\Mgb$.
We have a natural map
$$
\mu_K:\ocM_g(K)\la \Mgb(K);  \quad \quad \cC \mapsto \mu_\cC
$$
 where $\mu_{\cC}:\Spec K\to \Mgb$ is the moduli map  associated to $\cC$, already appeared  in Theorem~\ref{srt}.
 It is clear that the map $\red_K$ factors through $\mu$, i.e. we have
 \begin{equation}
\label{muK}
\red_K: \ocM_g(K) \stackrel{\mu}{\la} \Mgb(K) \la \Mgb  . 
\end{equation}
 \end{remark}
 
 \

So far, in Theorem~\ref{srt} and in Remark~\ref{Cs}, the valuation of  $K$, written  $v_K:K \to \RR\cup \{\infty\}$, did not play a specific role; it was  only used   to apply the valuative criterion of properness.
It will play a more important role in what follows.
 
The next statement  is a consequence  of  Theorem~\ref{srt}.  It can be found in various forms in  
 \cite{BPR2}, \cite{BPR1},  \cite{tyomkin}  and   \cite{viviani}.
 
\begin{prop}
\label{corrK} 
Let $\cC$ be a stable curve over $K$ and let $\cC_k$ be the stable curve over $k$ defined in  Remark~\ref{Cs}.
 Then
there exists an extended tropical curve $\Gamma_{\cC}=(G_{\cC}, \ell_{\cC}, w_{\cC})$ with the following properties.
 
\begin{enumerate}
 \item
$(G_{\cC},  w_{\cC})$ is the dual graph of $\cC_k$.
 \item
 $\Gamma_{\cC}$ is a tropical curve (i.e. all edges have finite length) if and only if   $\cC$ is smooth.
 \item
 If $ K'|K$ is a finite  extension and $\cC'$ the base change of $\cC$ over $K'$, then
  $\Gamma_{\cC}=  \Gamma_{\cC'}$.
 \end{enumerate}
\end{prop}
\begin{remark}
 As it will be clear from the proof, $\Gamma_{\cC}$ depends on $v_K$ even though our notation does not specify it.
\end{remark}
\begin{proof}
 The main point of the proof  consists  in completing the definition of  the tropical curve $\Gamma_{\cC}$ by explaining how the length function $\ell_{\cC}$ is defined. Using the notation of the previous theorem, for some finite  ring extension $R\subset R'$ we can assume that $\cC_k$ is the special fiber of a family of stable curves over      $R'$. Let $e$ be a node of $\cC_k$, then the equation of the family  locally at $e$ has the form
$xy=f_e$, with $f_e$ in the maximal ideal of $R'$.
We set 
$$
\ell_{\cC}(e)=v_{K'} (f_e)
$$
where $v_{K'}$ is the valuation    of the field $K'$.
Indeed, as $K$ is complete and the extension $K'|K$ is finite,
the valuation $v_{K'}$ is uniquely determined by $v _K$, moreover  $K'$ is complete;  see \cite[Prop. XII.2.5]{Lang}.
Now, if $\cC$ is smooth then for every node $e$ the function $f_e$ above is not zero, hence $\ell_{\cC}(e)\in \RR_{>0}$
and  $\Gamma_{\cC}$  is a tropical curve. Conversely, if $\cC$ has some node,
then this node specializes to some  node, $e$, of $\cC_k$. For such an $e$ the local equation, $f_e$, is equal to zero
 because the family is locally reducible. Therefore  $\ell_{\cC}(e) =\infty$, and 
part (2) is proved. It remains to prove that the definition of $\ell_{\cC}$ does not depend on the choice of the local equation or of the field extension. This is   standard, we refer to the above mentioned papers.
\end{proof}

  \begin{remark}
  \label{CCs}
Let $K$ be our complete valuation field with valuation $v_K$, and $\ocM_g(K)$ as defined in Remark~\ref{Cs}.
By what we said so far we can define a   {\it local    tropicalization map}, $\trop_K$, as follows
\begin{equation}
\label{tropl}
\trop_K: \ocM_g(K) \la \Mgtb;\quad \quad \cC \mapsto \Gamma_{\cC}.
\end{equation}

As done in Remark~\ref{Cs} for the map $\red_K$,  one easily checks   that     $\trop_K$ factors as follows
 \begin{equation}
\label{muKK}
\trop_K: \ocM_g(K) \stackrel{\mu}{\la} \Mgb(K) \la \Mgtb. 
\end{equation}
We abuse notation and denote, again, $\trop_K: \Mgb(K) \la \Mgtb$.
 \end{remark}

Now,  for every finite   extension $ K'|K$   
we have a map of sets given by base-change
$$
\beta_{K,K'}:\ocM_g(K) \la \ocM_g(K');\quad \quad \cC \mapsto  \cC\times_{\Spec K} \Spec K'.
$$
A  consequence of  Proposition~\ref{corrK} is that the local tropicalization map is compatible with finite base change, that is we have
the following.
 
\begin{cor} Let $K'|K$ be a  finite extension of complete valuation fields. 
Then 
 $\trop_K = \trop_{K'}\circ \beta_{K,K'}$.
\end{cor}
 
 \subsection{Tropicalization and analytification of $\Mgbst$}
 In the previous subsection we have shown that  to a local family of stable curves there corresponds an extended tropical curve, and this  is compatible with base change.
We can then ask for more,  namely for a
 global  version of this correspondence, where by ``global" we mean involving  the entire moduli spaces. That it ought to be possible to satisfy such a request    is indicated by another type of evidence, consisting in some  analogies between the moduli space of  tropical curves and the moduli space of stable curves, as we shall now illustrate.
 
As we have seen in \eqref{Mgt} and \eqref{Mgtb}, the moduli space of tropical curves is partitioned via stable graphs. Now, this is    true also
for the moduli space of stable curves, for which  we have 
$$
\Mgb = \bigsqcup _{(G,w) \in \SG}  M{(G,w)} 
$$
where $M{(G,w)}$ denotes the locus of stable curves having $(G,w)$ as dual graph (recall that $S_g$ is the set of   stable graphs of genus $g$).
The stratum  $M{(G,w)}$ turns out to be  an irreducible quasi-projective variety of dimension $3g-3-|E(G)|$,
 equal to the codimension of the locus $\Mt{(G,w)} $ in $\Mgt$.
Moreover, these strata can be  shown to form a partially ordered set with respect to inclusion of closures,  and we have:
$$
 M{(G,w)} \subset \overline{M_{(G',w')}}
\quad \Leftrightarrow \quad  (G,w) \to (G',w').
$$
This is analogous, though reversing the arrow,  to what happens in  $\Mgtb$, as we have seen in  \eqref{poset}.

Let us  go back to the purpose of this subsection:  to obtain  a global picture from which the local correspondence described earlier can be derived.
We proceed  following \cite{ACP} to which we refer for details.

The above described analogies     between $\Mgt$ and $\Mgb$, together with the construction of $\Mgt$ by means of Euclidean cones, 
  indicate  that  $\Mgt$ should be the ``skeleton"
of the moduli stack, $\Mgbst$, of stable curves.  We  must  now consider the stack, $\Mgbst$, instead of the scheme, $\Mgb$, because the former has a toroidal structure that the latter does not have. It is precisely the toroidal structure that enables us to
construct  the above mentioned skeleton as a   {\it generalized cone complex}  associated to $\Mgbst$, denoted by ${\Sigma}(\Mgbst)$,
and compactified by an   {\it extended generalized cone complex}, written $\ov{\Sigma}(\Mgbst)$. Having constructed these spaces one 
  proves that there  are isomorphisms  of generalized cone complexes  represented by   the vertical arrows of the following commutative diagram (see \cite[Thm. 1.2.1]{ACP})
  
  \
  
\begin{equation}
 \label{sigma}
\xymatrix@=.5pc{
   &&&\Sigma(\Mgbst) \ar@{^{(}->}[rrr]  \ar[ddd]_{\cong}&&&{\ov \Sigma}(\Mgbst)  \ar[ddd] ^{\varphi}_{\cong}&&\\
&&&&&   && \\
&&&&&   && \\
   &&&\Mgt \ar@{^{(}->}[rrr]  &&&\Mgtb  &&&&  \\
 &&&  &&&&&& && } 
\end{equation}
The above diagram offers an interpretation, from a broader context,  of  the  global analogies between the moduli spaces of stable and tropical curves 
illustrated earlier.

We now turn to  the local correspondences described in the previous subsection   in order to place them into a more general framework.
To do that, we introduce the {\it Berkovich analytification}, $\Mgban$, of $\Mgb$.  

In this paper, we cannot afford to introduce analytic geometry over non-Archimedean fields  in a satisfactory way,
despite  its deep connections  
to tropical geometry of which  a suggestive example 
  is the   application to  the moduli theory of curves we are about to describe.
On the one hand this would require a consistent amount of space, on the other hand it
 would take us too far from our main topic. 
 
The starting point is that, following   \cite[Chapt. 3]{berkovich}, to any scheme  $X$  over $k$ one associates its   analytification, $X^{\an}$,
  a   more suitable space  from the analytic point of view.

We apply this  to the scheme $\Mgb$, and   limit ourselves to mention  that $\Mgban$ is a Hausdorff, compact
topological space, and to     describe it  set-theoretically.  

To do that, 
for any  non-Archimedean field $K|k$ consider the set,     $\Mgb(K)$,  of $K$-points of $\Mgb$.
Then  we have
\begin{equation}
 \label{Mgban}
\Mgban=\frac{\bigsqcup _{K|k} \Mgb(K)}{\sim_{\an}}
\end{equation}
where the union is over all non-Archimedean extensions $K|k$, and the equivalence relation $\sim_{\an}$
is defined as follows. Let  $\xi_1$ and $\xi_2$ be  $K_i$-points of $\Mgb$, i.e.   $\xi_i:\Spec K_i\to \Mgb$ for $i=1,2$.
We set $\xi_1\sim_{\an} \xi_2$ if there exists a third extension $K_3|k$ which extends also
$K_1$ and $K_2$, and a point $\xi_3\in  \Mgb(K_3)$ such that the following diagram is commutative
 $$\xymatrix@=.5pc{
 &&&&&&\Spec K_3\ar[ddrr]\ar[ddll]\ar[dddd]^{\xi_3}&&\\
  &&&&&   && \\
&&&& \Spec K_1 \ar[ddrr]_{\xi_1}&&&&\Spec K_1  \ar[ddll]^{\xi_2}&&&&  \\
  &&&   &&&&&&&\\
&&&&&&\Mgb&&&&&&&
}$$
We conclude by recalling that \eqref{Mgban} induces a bijection between $\Mgban(K)$ and $ \Mgb(K)$ for every $K$; see \cite[Thm 3.4.1]{berkovich}.

\begin{remark}
 From the above description, a  point  of $\Mgban$
is represented by a   stable curve $\cC$ over a non-Archimedean field  $K$.
By Theorem~\ref{srt} we can furthermore assume (up to field extension) that $\cC$ admits a stable model over the valuation ring of $K$.
\end{remark}
Now,  we can use   the local tropicalization maps defined in \eqref{tropl}  to  define  a map
$$
\trop: \Mgban \la \Mgtb 
$$
such that 
 its  restriction  $\Mgban(K)= \Mgb(K)$ coincides with the map $\trop_K$ defined in Remark~\ref{CCs}.

The connection between this map   and Diagram \eqref{sigma} is achieved using results of \cite{thuillier}, which enable us to construct a retraction (or projection)  from the analytification of $\Mgb$ to the extended skeleton of $\Mgbst$. This remarkable retraction map is denoted by ``${\operatorname{p}}$'' in the commutative diagram below.

Finally, let us introduce 
the two forgetful maps from $\Mgb$ and $\Mgtb$ to   the set $S_g$ of stable graphs of genus $g$
$$
\gamma_g:\Mgb \la S_g, \quad \quad  \quad \gamma_g^\trop:\Mgtb \la S_g
$$
mapping a stable, respectively   tropical,  curve to its dual,
 respectively  underlying,  weighted graph.

Now, 
using the  above notation, 
 we summarize  the previous discussion with the following statement.
\begin{thm}
The following  diagram   is commutative. 
 $$\xymatrix@=.5pc{
   &&&\Mgban \ar [rrrr]^{\operatorname{p}}  \ar[dddrrrr]_{\trop} \ar[ddd]_{\operatorname{red}}&&&&{\ov \Sigma}(\Mgbst)  \ar[ddd] ^{\varphi}_{\cong}&&\\
&&&&&   && \\
&&&&&   && \\
   &&&\Mgb\ar [rrd]_{\gamma_g}  &&&&\Mgtb\ar [lld]^{\gamma_g^\trop}   &&&&  \\
 &&&  &&S_g&&&& && }$$
\end{thm}
Notice  that all the arrows in the diagram are surjective.

With respect to the previous subsections, the   new and non-trivial part of the diagram is its upper-right corner,   a special case of \cite[Thm.1.2.1]{ACP}, to which we refer for a proof.

We   need to  define the ``reduction" map,  ${\operatorname{red}}:\Mgban\la \Mgb$.
Using the set-up of the previous subsection, this  map sends a point in $\Mgban$ represented by a stable curve $\cC$ over the   field $K$ to the stable curve $\cC_k$, defined  over $k$, of Remark~\ref{Cs}.
The commutativity of the diagram follows immediately from the earlier discussions.

    \section{Curves and their Jacobians}
    \label{CurvesSc}
  Our main focus in this paper is the role of tropical and graph-theoretical methods in the theory of algebraic curves and their moduli. As is well known, the geometry of algebraic curves is closely  connected to the geometry of their Jacobian varieties,
 to study which    combinatorial  methods
  have been used for  a long time. This will be the topic of the present section.
   
     \subsection{Jacobians of smooth curves and their moduli}
To any smooth connected projective curve $C$ of genus $g$  defined over $k$ one associates its  Jacobian,  $\Jac(C)$,
an abelian variety   of dimension $g$ endowed with a canonical principal polarization, the theta divisor, $\Theta(C)$.
The pair $(\Jac(C), \Theta(C))$ is a so-called {\it principally polarized abelian variety}, and will be denoted
\begin{equation}
 \label{jacCT}
 {\bf{Jac}}(C):=(\Jac(C), \Theta(C)) 
\end{equation}

The structure of  ${\bf{Jac}}(C)$      is extremely rich and    a powerful instrument to    study the curve $C$; it   has been investigated  in depth
for a long time and from different points of view: 
algebro-geometric, arithmetic and  analytic. 

There exist  various  ways of describing the Jacobian,  one of which is
\begin{equation}
 \label{jacC}
  \Jac(C)=\Pic^0(C),
\end{equation}
that is, $  \Jac(C)$ is the moduli space of line bundles (equivalently, divisor classes) of degree $0$ on $C$.
Now, we describe the Theta divisor   in a way that will be useful later.
First, we identify 
$$
\Pic^0(C) \cong \Pic^{g-1}(C)
$$
(even though the isomorphism, mapping $L\in \Pic^0(C) $ to $L\otimes L_0$ for some  fixed $L_0\in \Pic^{g-1}(C)$,
 is not canonical),   then we identify $\Theta(C)$ as follows
\begin{equation}
 \label{theta}
\Theta(C)  =\{L\in \Pic^{g-1}(C):\  h^0(C,L)\geq 1\}.
\end{equation}
It is well known that $\Theta(C)$ is an irreducible, codimension-one closed subvariety of $\Pic^{g-1}(C)$, and an ample principal divisor.

If the base field  $k$ is   $\C$, then   $\Jac(C)$ is identified with a $g$-dimensional  torus, as follows
\begin{equation}
 \label{jacCC}
   \Jac(C) = H^1(C, \cO_C)/H^1(C, \Z)\cong \C^g/\Z^{2g}.
\end{equation}
  Using the notation \eqref{jacCT} we recall the  following  famous      Torelli Theorem.

\begin{thm}Let $C_1$ and $C_2$ be two smooth curves; then
  $C_1\cong C_2$ if and only if $ {\bf{Jac}}(C_1)\cong {\bf{Jac}}(C_2).$
\end{thm}

What about the moduli theory of Jacobians? 

Since  Jacobians are    abelian varieties of a special type, it is   natural to approach this issue 
 within the moduli theory of principally polarized abelian varieties.
In fact, this theory has been developed to a  large extent in parallel to the moduli theory of smooth curves,
and the broader area has seen  an extraordinary development in the second half of the twentieth century.
We have already discussed the case of curves, let us briefly discuss the case of abelian varieties.

The moduli space of principally polarized abelian varieties of dimension $g\geq 2$ is denoted by
$ 
A_g.
$ 
It is an irreducible algebraic variety of dimension $g(g+1)/2$.  
The set of all Jacobians of curves is a distinguished   subspace  of   $A_g$, the so-called Schottky locus, written $Sch_g$.
The   Torelli theorem can   be re-phrased in this setting by saying that the following  map    

$$
\tau: M_g  \la A_g;\quad \quad C\mapsto {\bf{Jac }}(C)
$$
is an injection. In fact $\tau$, the {\it Torelli map}, is a morphism of algebraic varieties whose
  image   is, of course,  the Schottky locus.

 Just like $M_g$, the space $Sch_g$ is not projective and one is interested in 
 constructing   useful compactifications of it. Now,   $A_g$    is not projective either.
So, one could  compactify $A_g$ first, and then study the closure of $Sch_g$; this can be done,
as we shall see in Subsection~\ref{CJsec}.
On the other hand, the problem of completing just the space of all Jacobians is somewhat special and can  be dealt with independently, as we are going to explain.
 
 A natural approach  is to start  from the compactification,  $\Mgb$,
 of $M_g$ by stable curves, and try to describe,  in terms of
 stable curves, the  points to add to complete $Sch_g$.

  With this in mind, the first object to   study is the (generalized)  Jacobian  of a stable curve. 
  
    \subsection{Jacobians of nodal curves}
    
Let $X$ be a nodal curve defined over $k$
and $G_X=(V,E)$ its dual graph. Its (generalized) Jacobian, $\Jac(X)$, is defined   extending  \eqref{jacC}, as the group of line bundles on $X$ having degree $0$ on every irreducible component of $X$.
It is well known  that $\Jac(X)$ is a connected algebraic variety  and a commutative algebraic group,
but it is not  projective, i.e. it is not an abelian variety,  with   the exception of a few cases to be described in Remark~\ref{ct}. 

More precisely, we are going to show that  $\Jac(X)$ is a {\it semi-abelian} variety, i.e. it fits into an exact sequence of type  
\begin{equation}
 \label{semi}
0 \la T\la \Jac(X) \la A \la 0
\end{equation}
where $A$ is an abelian variety  and $T\cong (k^*)^b$ with $b\in \Z_{\geq 0}$.
The above sequence for our   curve $X$   can be recovered as follows.
Recall that   the normalization of $X$ is written  $X^{\nu}=\sqcup _{v\in V} C_v^{\nu}$, and the normalization map is
$ 
\nu:X^{\nu} \to X.
$  
The sequence   \eqref{smap} can be restricted to the subsheaves of units,
$$
0\la \cO^*_X\la  \nu_*\cO^*_{X^{\nu}}\la \cS^*  \la 0 
$$
yielding  
 the  following cohomology sequence, where we identify  $H^1(... ,\cO^*_{...})$ with $\Pic(...)$,
$$ 
0\to H^0(X,\cO^*_X)\to H^0(X^{\nu},\cO^*_{X^{\nu}}) \to\ (k^*)^{|E|}\to
\Pic(X) \stackrel{\nu^*}{\la}\Pic ( X^{\nu})\to 0.
$$
Now,  a line bundle $L\in \Pic(X)$ (or in $\Pic(X^{\nu})$) has a {\it multidegree} $\md \in \Z^V$ such that $d_v=\deg_{C_v}L$ 
for all $v\in V$. Therefore
 in the above sequence, the two Picard schemes decompose into connected components, one for every multidegree,
 as follows (we do that only for $X$, the case of  $X^{\nu}$ is similar) 
 $$
 \Pic(X)=\sqcup_{\md\in  \Z^V}\Pic^{\md}(X)
 $$
 where $\Pic^{\md}(X)$ denotes the subscheme of line bundles of multidegree $\md$.

The  Jacobian varieties correspond to the multidegree $(0,\ldots, 0)$,
therefore by restricting the sequence to the Jacobians   we get
$$
0\to H^0(X,\cO^*_X)\to H^0(X^{\nu},\cO^*_{X^{\nu}}) \stackrel{\hat{\delta}}{\la}(k^*)^{|E|}\to
\Jac(X) \stackrel{\nu^*}{\la}\Jac( X^{\nu})\to 0.
$$
With the notation of the sequence \eqref{semi}, we set
$$
A=\Jac(X^{\nu})= \Pi_{v\in V} \Jac(C_v^{\nu})
$$
and we have identified the right   side of   Sequence \eqref{semi}. Now,  using what we proved in Subsection~\ref{gtp}, we see
that the map $\hat{\delta}$ in the cohomology sequence above can be   identified with the   coboundary map on the graph $G_X$, and we have the following exact sequence 
$$
0\la  H_1(G_X, k^*) \la \Jac(X) \la \Jac (X^{\nu}) \la 0.
$$

In conclusion,   we have an identification 
$$
T=\Ker  \ \nu ^*  = H_1(G_X, k^*) = (k^*)^{b_1(G_X)}
$$
completing the explicit description of \eqref{semi}.
 \begin{remark}
 \label{ct}
From the above discussion we easily derive   that $\Jac(X)$ is projective (i.e. an abelian variety) if and only if 
$b_1(G_X) =0$, i.e. if and only if $G_X$ is a tree (the weights on the vertices play no role). If this is the case  $X$ is said to be a curve of {\it compact type}.
\end{remark}
 
\subsection{Models for Jacobians over a DVR}
  We have described Jacobians for a fixed curve, now we want
to consider families of curves and their corresponding families of Jacobians, and   concentrate on   the fact that
such families are seldom complete.  

Let  $R$ be a  discrete valuation ring whose fraction field is written $K$. In this subsection we can drop the assumption that $K$ is complete. 
Now let
$$
\phi:\cC_R\to \Spec R
$$ be a family of  curves  with smooth  generic fiber;
denote
  by $\cC_k$ the  special fiber of $\phi$ and   assume    $\cC_k$
 is a nodal curve over $k$.  We  write
$ 
 \cC_k \cong X
$ 
to tie in with the notation in the previous subsection, which we continue to use.
We   then consider the associated relative Jacobian, which is a morphism
$$
\cJ_R\la \Spec R
$$
such that the generic fiber, written $\cJ_K$, is the Jacobian of the generic fiber of $\phi$,  and the special fiber, $\cJ_k$, is  the Jacobian of $X$.
The above morphism is smooth and separated  but, as   explained in Remark~\ref{ct},  is not proper 
 unless $X$ is a curve of compact type.  
Furthermore, it is   ``too small" from the point of view of the moduli theory of line bundles.
More precisely, recall that the Jacobian of a smooth curve is the moduli space for line bundles of degreee $0$ on the curve, i.e. 
subvarieties of the Jacobian are in bijective correspondence with families of line bundles  (of degree $0$) on the curve.
Now, this fails for $\cJ_R$, and  most families of degree $0$ line bundles on the  fibers of $\cC_R$
do not correspond to subvarieties of $\cJ_R$.

\begin{example}
 Suppose the special fiber, $\cC_k$, of our family has two irreducible components    and consider a line bundle $\cL$ on $\cC_R$ such that its restriction to $\cC_k$ has multidegree $(1,-1)$. Our $\cL$ gives a family of degree $0$ line bundles on the fibers of $\phi$ which does not correspond to any subvariety of $\cJ_R$, simply  because its bidegree on the special fiber is not $(0,0)$.
\end{example}

To remedy these problems one concentrates on  the generic fiber, $\cJ_K$, and asks whether it admits a model 
over $\Spec R$ with   better properties than $\cJ_R$.

The following commutative diagram summarizes the current state of our knowledge with respect to models  that are ``canonical",
i.e. independent of any  specific choice. All arrows represent morphisms over $R$, the three arrows on the right  are smooth morphisms.

In the diagram, referring to   \cite{neron}, \cite{raynaud}, and  \cite{BLR} for details, we have:

$\bullet$ $N(\cJ_K)\to \Spec R$ is the N\'eron model of $\cJ_K$; it is a group scheme, separated over $R$
but not complete, in general.

$\bullet$  $\Pic^0_\phi \to \Spec R$ is the relative degree-$0$ Picard scheme of the original family of curves, 
$\phi:\cC_R\to \Spec R$. Its special fiber is the group of all line bundles of degree $0$ on $X$. It is,   in general, neither
quasi-projective nor separated, but it is a moduli space for line bundles of degree $0$ on the fibers of $\phi$.

$\bullet$ The surjection $\Pic^0_\phi \to N(\cJ_K)$ can be described as the  largest separated quotient of  $\Pic^0_\phi$ over $\Spec R$. We notice in passing that it does admit    sections.

\begin{equation}
\xymatrix{
&&\cJ_R \ar@{^{(}->}[d]\ar@{ ->>}[rrd]&& \\
\cJ_K \  \ar@{^{(}->}[rr] \ar@{^{(}->}[rrd] \ar@{^{(}->}[rru]  &&N(\cJ_K)\ar@{ ->>}[rr] &&\Spec R \\
&&\Pic^0_\phi  \ar@{ ->>} [u] \ar@{ ->>}[rru]&&
}
\end{equation}

What about the special fibers of the schemes appearing in the  middle of the diagram?
We already know the special fiber of $\cJ_R$, of course; what about the other two?

  The special fiber of    $\Pic^0_\phi $ is a disjoint union of copies of $\cJ_k$, the Jacobian of $X$. To describe it explicitly,
for any  $\md \in \Z^V$
  consider  $\Pic^{\md}(X)$. We have non canonical isomorphisms
$$
\Pic^{\md}(X) \cong   \cJ_k.
$$
Now, the special fiber of the Picard scheme $\Pic^0_\phi\to \Spec R$  is
$$
\Pic^0(X)=\bigsqcup _{\md\in \Z^V: \sum_{v } d_v = 0}\Pic^{\md}(X).
$$
Hence
$\Pic^0(X)$ is an  infinite   union of copies of $\cJ_k$, unless    $X$ is irreducible in which case  $\Pic^0(X)$ is also irreducible.

   To describe the special fiber, $N_k$,  of the N\'eron model $N(\cJ_K)$, we shall make an additional    assumption,
 namely that the total space of the family of curves, $\cC_R$, is nonsingular (see Remark~\ref{Nns} below).
We have
 $$
N_k\cong \bigsqcup _{i\in \Phi_{G_X}} \cJ_k^{(i)},
$$
in other words,  $N_k$ is, again, a disjoint union of copies of the Jacobian of $X$. These copies are 
  indexed by   an interesting finite group, $\Phi_{G_X}$, which  has been  object of study   since N\'eron models were 
constructed in \cite{neron} and,   for Jacobians, in \cite{raynaud}. As the notation suggests, 
 $\Phi_{G_X}$ depends only on the dual graph of $X$. Indeed
 we have    
\begin{equation}
 \label{kt}
\Phi_X=\Phi_{G_X}\cong {\partial C_1(G_X , \Z)\over \partial \delta C_0(G_X, \Z)},
\end{equation}
where, for some orientation on the graph $G_X$ (the choice of which is irrelevant)
  ``$\delta$" and ``$\partial$"  denote the coboundary and boundary maps  defined in subsection ~\ref{gtp};
  see \cite{OS}.
 See \cite{raynaud}, \cite {BLR}, \cite{lorenzini}, \cite{BMS1} and  \cite{BMS2}   for more details on the group  $\Phi_X$.
  
\begin{remark}
 \label{Nns} Suppose the total space $\cC_R$ is singular. Then  its desingularization,  $\cC_R'\to \cC_R$,  is an isomorphism away from 
 a set, $F$, of   nodes of the special fiber, $X$,  and the preimage of every node in $F$ is a chain of rational curves. The new special fiber is thus a nodal curve, $X'$, whose dual graph is obtained by inserting some vertices in the edges of $G_X$ corresponding to $F$. One can show that if $F$ is not entirely  made of bridges of $G_X$, then the group $\Phi_{X'}$ is not   isomorphic to $\Phi_{X}$.
 
\end{remark}

\subsection{Compactified Jacobians and Torelli maps}
\label{CJsec}
In the previous section we described N\'eron models for Jacobians. Although they
have many good properties (they are canonical and they have a universal mapping property, see \cite{BLR}), their use is limited to certain situations. First of all, they are defined over one-dimensional bases
(in fact in the previous subsection we assumed $R$ was a discrete valuation ring). Secondly, they are not projective.

The   problem of
constructing projective models for Jacobians, or  compactifications of Jacobians for singular curves, 
   has been investigated for a long time, also before stable curves were introduced and $M_g$ compactified. 
 
Building upon the seminal work of Igusa, Mumford,  Altman-Kleiman and  Oda-Seshadri,  substantial progress was made in the field and, since the beginning of the twenty-first century, 
we have a good understanding of compactified Jacobians of stable curves and compactified moduli spaces for Abelian varieties.

Compactified Jacobians for families of stable curves were constructed in  \cite{Cthesis},  \cite{Simpson}, \cite{Ishida}.
Some of these  constructions have been compared to one another in \cite{AlexeevTorelli} and related to the compactification of the moduli space of abelian varieties constructed in \cite{Alexeev}.

Let us turn to the Torelli map $\tau$ defined earlier for smooth curves, and   study how to  extend it over the whole of $\Mgb$ in a geometrically meaningful way.  This problem is a classical one, and has been studied for a long time;  see \cite{NamikawaII} for example.  Let us present a solution   in the following diagram.
$$
 \xymatrix{  
 \Mgb \ar[r] ^{{  \overline{\tau}}}& \ \ {\overline{A_g}};      
 &[X] \ar@{{|-}->}[r]&[{\bf\overline{Jac}}(X)]= \bigr[ \Jac(X) \curvearrowright (\PX,\overline{\Theta(X)})\bigl]
&&&& \\
 M_g \ar@{^{(}->}[u]      \ar[r]^ {\tau}&  A_g\ar@{^{(}->}[u]     &&&&\\
 }
 $$
In order  to define the extension ${\overline{\tau}}$ we followed  \cite{AlexeevTorelli}  and  used    the main irreducible component, $\overline{A_g}$,  of the compactification   of the  moduli space of principally polarized abelian variety constructed in \cite{Alexeev};   this is   a moduli space for
so-called ``semi-abelic stable pairs"   (which we shall not define here). Clearly, the image of  $\overline{\tau}$ contains and compactifies the Schottky locus, $Sch_g$, defined earlier.
\begin{remark}
\label{redAK}
 For future use, in analogy with Remark~\ref{Cs}, we denote by $\cA_g(K)$ the set of principally polarized abelian varieties of dimension $g$ defined over the field  $K$, and  define the  {\it reduction map}
$$
\red_K^{A_g}: \cA_g(K) \la  {\overline{A_g}}. 
$$
The map $\red_K^{A_g}$  sends a principally polarized abelian variety $\cP$ defined over $K$ to the image of the special point of the 
map $\Spec R \to {\overline{A_g}} $, extending the moduli map  $\Spec K \to A_g$  associated to $\cP$.
\end{remark}

For a stable curve $X$,  let us describe its image via $\overline{\tau}$, written $[{\bf\overline{Jac}}(X)]$ in the above diagram.  This is    the isomorphism class of the 
pair $(\PX,\overline{\Theta(X)})$ acted upon by the group $\Jac(X)$. Here $\PX$  denotes  the  compactified Jacobian  constructed 
  in  \cite{Cthesis}, namely  the moduli space for  so-called ``balanced" line bundles of degree $g-1$ on semistable curves   stably equivalent to $X$.  The variety $\PX$  is   connected, reduced,  of pure dimension $g$.
 Next, $\overline{\Theta(X)}$, called again the {\it Theta divisor}, is defined 
in analogy with \eqref{theta} as the closure in $\PX$ of the locus of line bundles on $X$ having some non-zero section.
$\overline{\Theta(X)}$ is an ample   Cartier divisor; see   \cite{esteves} and    \cite{Ctheta}.
We shall describe ${\bf\overline{Jac}}(X)$  in more   details in subsections~\ref{CJsec} and \ref{torellisec}.

 In this set-up we   propose the Torelli problem for stable curves, as follows:
describe the fibers of the morphism $ \overline{\tau}$, and in particular   the loci where it is injective.
That $ \overline{\tau}$ could not possibly be  injective is well known; see \cite{Namikawa} for example. 
Let us focus    on this aspect. First of all,  we have
$$
 \overline{\tau}^{-1}(A_g)=\Mgb^{\operatorname{cpt}}
$$
where $\Mgb^{\operatorname{cpt}}$ is the locus in $\Mgb$ of curves of compact type. Indeed, we have seen in  Remark~\ref{ct}  that the generalized Jacobian of a curve of compact type is an abelian variety.
Now we have the following well known fact.
\begin{prop}
\label{ctprop}
 Let $X_1$ and $X_2$ be stable curves   of compact type.
If  $X_1^{\nu}\cong X_2^{\nu}$ then $ \overline{\tau}(X_1)=\overline{\tau}(X_2)$.
The converse holds if $X_1$ and $X_2$ have the same number of irreducible components. \end{prop}
\begin{proof}
 Let us begin by proving  the first part  in the simplest nontrivial case, of a curve $X=C_1\cup C_2$  with only two (necessarily smooth) components.
 Let $g_i$ be the genus of $C_i$, so that $g_i\geq 1$ by the stability assumption.
 
 As we said above, we can identify
 $ \overline{\tau}(X)$ with the image of $X$ via the classical Torelli map $\tau$.
As already shown, we have 
 $$
 \Jac(X)\cong \Jac(X^{\nu})= \Jac(C_1)\times \Jac(C_2).
 $$ 
It remains to look at the theta divisor. It is well known that, $\Theta(X)$  is the union of two irreducible components   as follows
  $$
 \Theta (X)\cong  \Bigr(\Theta (C_1)\times \Jac(C_2)\Bigl) \cup \Bigr(\Jac(C_1)  \times \Theta (C_2)\Bigl),
 $$
 hence  $\Theta (X)$ depends only on $X^{\nu}$, as claimed.
 
 Now, if $X$ has more than two components the proof is similar.
It suffices to observe that, writing as usual  $X^{\nu}=\cup_{v\in V}C_v$, the components $C_v$ are   smooth,  hence
$$ 
 \Jac(X)\cong \prod_{v\in V}\Jac(C_v)=\prod_{v:g_v>0}\Jac(C_v),
$$
where $g_v$ is the genus of $C_v$.
For the Theta divisor we have
$$
 \Theta (X)\cong \bigcup_{v:g_v>0}\Bigr(\Theta (C_v)  \times\prod_{w\in V\smallsetminus \{v\}}\Jac(C_w)  \Bigl)
$$
so  $ \overline{\tau}(X)$  only depends on the normalization of $X$.

For the converse, by what we observed above and by  the Torelli theorem for smooth curves,  the Jacobian of $X$  detects precisely    the components of positive genus   of $X$, which are therefore the same for  $X_1$ and $X_2$. 
Hence, as   $X_1$ and $X_2$ have the same number of components,   they also have
  the same number of components of   genus zero, so we are done.
\end{proof}

\subsection{Combinatorial stratification of the compactified Jacobian}
\label{CJsecc}
Let $X$ be a stable curve and $(G_X,w_X)$ its dual graph.
We are   going to describe an  interesting stratification of $\PX$ 
preserved   by the action of $\Jac(X)$.
The strata depend on certain sets of nodes of $X$,
to which we refer as the strata  {\it supports},
 so that each stratum is isomorphic to the Jacobian of the desingularization of $X$ at the corresponding support. 
As we are going to see, this stratification is purely combinatorial, governed by the dual graph.
 
First of all, for
  any graph $G=(E,V)$,  we define the     partially ordered set  of {\it supports},   $\SP_{G}$, as follows.
\begin{equation}
 \label{SP}
\SP_{G}:=\{S\subset E:\  G-S \text{ is free from bridges}\}, 
\end{equation}
where a {\it bridge}, or {\it separating edge},  is an edge not contained in any cycle (equivalently,  an edge whose removal disconnects the connected component  containing it); notice that $G-S$ needs not be connected.
  $\SP_{G}$ is partially ordered by inclusion: $S\geq S'$ if $S \subset S'$.

Now, for any $S\in \SP_{G_X}$ we denote by $X_S^{\nu}$ the curve obtained from $X$ by desingularizing the nodes in $S$. 
The  dual graph of $X_S^{\nu}$ is thus  $(G_X-S, w_X)$. On the curve $X_S^{\nu}$   we have a distinguished  finite  set, written  $\Sigma_S$, of so-called ``stable" multidegrees for line bundles of total degree $g(X_S^{\nu})-1$; we   omit the precise definition as it is irrelevant here.

We are ready to introduce the stratification of $\PX$:
 
\begin{equation}
 \label{PX}
 \PX=\bigsqcup_{\stackrel{S\in \SP_{G_X}}{\md \in \Sigma_S}} P^{\md}_S  
\end{equation}
and 
for every stratum $P^{\md}_S$   above we have a canonical isomorphism
$$
P^{\md}_S\cong \Pic^{\md}X_S^{\nu}.
$$
Now,  the Jacobian 
  $\Jac(X)$   parametrizes line bundles of degree zero on every component of $X$.
Its  action  on $\PX$ can be described, in view of the above isomorphisms,
   by  tensor product of line bundles pulled-back to the various normalizations, $X_S^{\nu}$.
   Therefore we have the following important fact.

\begin{remark}
\label{orbit}
 The strata $P^{\md}_S$ appearing in \eqref{PX} coincide with the orbits of the action of $\Jac(X)$.
\end{remark}

We denote by $\ST_X$ the set of all strata, i.e. $\Jac(X)$-orbits, appearing in \eqref{PX}.  Remark~\ref{orbit} 
implies that 
the  closure of any $P_S^{\md}\in \ST_X$
 is a union of strata in $\ST_X$. In particular, $\ST_X$ is partially ordered as follows:
$$
P_S^{\md}\geq P_{S'}^{\md'} \quad \text {if } \quad  P_{S'}^{\md'} \subset \overline{P_S^{\md}}.
$$
Now, it is a fact that the following map, associating to a stratum its support,
\begin{equation}
 \label{Supp}
\operatorname{Supp}:\ST_X\la \SP_{G_X}; \quad \quad P_S^{\md}\mapsto S
\end{equation}
 is surjective and order-preserving (it is a morphism of posets).
So, in order to understand the geometry of $\PX$, we turn   to   the combinatorial properties of $\SP_{G_X}$,
for which we need a detour into graph theory.
A pleasant surprise is awaiting us: our detour will naturally take us to study Jacobians of graphs and tropical curves,
and their Torelli problem.

   \section{Torelli theorems}
   \label{TorelliSc}
In this section we discuss   the Torelli problem for  graphs, tropical curves, and stable curves, by respecting   the time sequence  in which these   problems were solved,    using many of the same combinatorial techniques. We devote    Subsection~\ref{twsec} 
  to the Torelli problem for weighted graphs, which we prove as an easy   consequence of the Torelli theorem for tropical curves and which, to our knowledge, does not appear in the literature.

\subsection{Some algebro-geometric graph theory}

In  the early 90's, a line of research in graph theory   began establishing an analogue of the classical Riemann-Roch theory, viewing a finite graph as the analogue of a Riemann surface;  see \cite{biggs}, \cite{BdlHN}, \cite{BN}, for example.
So, a divisor theory for graphs was set-up and developed in that direction;  we will now give    just   a short overview.

The group of all divisors on a graph  $G=(V,E)$, denoted by $\Div(G)$, is defined as the free abelian group
on $V$. Hence $\Div(G)$ can be naturally identified 
  with $\Z^V$;  the degree of a divisor   is   the sum of its $V$-coordinates.
  
  In analogy with algebraic geometry one defines ``principal" divisors as divisors associated to ``functions". 
A   function on $G$ is  defined as a map
  $ 
  f:V\to \Z,
  $ 
hence every function  has the form
 $
f=\sum_{v\in V}c_vf_v
 $,
 where $c_v\in \Z$ and $f_v$ is the function such that  
  $ f_v(w)=\kappa_{v,w}
$ 
  for any $w\in V$. The set of all functions on $G$ is the free abelian group generated by $\{ f_v,\  \forall v\in V\}$.
 
 We can now define, for any  function $f$ as above, the {\it principal} divisor, $ \dv(f)$, associated to $f$.
  By what we said it suffices to define $ \dv(f_v)$ for every $v\in V$ and extend the definition by linearity. We set
  
  $$
 \dv(f_v)=- \val(v)+2\operatorname{loop}(v) +\sum _{w\in V\smallsetminus v} (v, w)w
  $$
  where  $\val(v)$ is the valency of $v$,  $\operatorname{loop}(v)$ is the number of loops based at $v$, and $(v, w)$ the number of edges between $v$ and $w$.  
  
\begin{example}
 Let  $f$ be a constant function, i.e. $f=c\sum_{v\in V}f_v$ for some $c\in \Z$.
 Then $\dv(f)=0$.
\end{example}
  See Example~\ref{Z3} for more examples.

Principal divisors have       degree   zero  and form a subgroup, written  ${{\operatorname{Prin}} ( G)}$, 
  of  $\Div^0(G)$, the group of divisors of degree $0$ on $G$.

Two  divisors  are linearly equivalent if their difference is in  $\Prin (G) $.
The quotient $ \Div^0(G)/ \Prin  (G)$
is a finite group called, not surprisingly with \eqref{jacC} in mind,  the {\it Jacobian group} of $G$. 

The Jacobian group turns 
out to be isomorphic to a finite group  we  encountered earlier, namely we have 
$$
 {{\Div^0(G)}\over{{\operatorname{Prin}}  (G)}}   \cong  \Phi_G. $$

The above isomorphism follows from \eqref{kt} together with the following simple observation  
 $$ 
 {{\Div^0(G)}\over{{\operatorname{Prin}}  (G)}}   \cong 
 {\partial C_1(G  , \Z)\over \partial \delta C_0(G , \Z)}.$$
Indeed, we can naturally identify $\Div(G)=C_0(G,\Z)$, and then we have $\partial  C_1(G  , \Z)\subset \Div^0(G)$;
now, as $G$ is connected, this containment is easily seen to be an equality.
Next, we have
$$
\partial \delta (v)=\partial \Bigr(\sum_{t(e)=v}e-\sum_{s(e)=v}e\Bigl) =  \val(v)-2\operatorname{loop}(v) -\sum_{w\in V\smallsetminus v}(v,w)w =-\dv(f_v)
$$
proving that  ${\operatorname{Prin}}  (G) $ is identified with $\partial \delta C_0(G , \Z)$.

Several   results have been obtained in this area, the   Riemann-Roch theorem among others; we refer to \cite {BN} for weightless graphs, and to \cite{AC} for the case of weighted graphs. 

By contrast, an analog of the Torelli theorem for graphs does not find its place in this setting.
Indeed, the    problem  is that the   group $\Phi_G$    is not
   a  good  tool to characterize graphs, as it is easy to find
 examples of different graphs having the same Jacobian group.

\begin{example}
\label{Z3}
Let us show that all  graphs in Figure 5   have   $\Z/3\Z$  as Jacobian group.

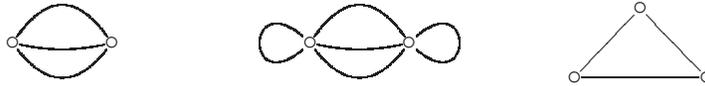
\begin{figure}[h]
\begin{equation*}
\xymatrix@=.5pc{
&&&&&&&&&&&&&&&&&&&&&*{\circ}\ar@{-}[ddrr]\ar@{-}[ddll]&&&&&\\
&&*{\circ}\ar@{-} @/^ 1.2pc/[rrr] \ar@{-} @/_.2pc/[rrr] \ar@{-} @/_1.1pc/[rrr]   &&&*{\circ}  
&&&&&&
  *{\circ}\ar@{-}@(ul,dl)
\ar@{-} @/^ 1.2pc/[rrr] \ar@{-} @/_.2pc/[rrr] \ar@{-} @/_1.1pc/[rrr]   &&&*{\circ}  \ar@{-}@(ur,dr)&&&&&& \\ 
&&&&&&&&&&&&&&&&&&&*{\circ}\ar@{-}[rrrr] &&&&*{\circ}  &&  \\
 & &&&&&}
\end{equation*}
\caption{Graphs with   isomorphic Jacobian group.}
\end{figure}
It is clear   that the definition of the Jacobian group does not take loops into account, hence the first two graphs have isomorphic Jacobian groups.
Now, that this group is   $\Z/3\Z$ is easily seen  as follows. Denote by $G$ the graph on the left, then
$$
\Div^0(G)\cong \{(n, -n),\ \forall n\in \Z\}  \quad  \quad \quad \Prin(G)\cong \{(3n, -3n),\  \forall n\in \Z\}; 
$$
indeed, if $v,w$ are the two vertices of $G$, then $\dv (f_v)=-3v+3w$ and $\dv (f_w)=- \dv (f_v)$.

Next, denote by $H$ the ``triangle" graph on the right. We have
$$
\Div^0(H)\cong \{(n, m,-n-m),\ \forall n,m\in \Z\}  \cong \Z\oplus \Z$$
and (recalling that $\dv(\sum_{v\in V}f_v)=0$)
$$
\Prin(H)\cong \langle \{(1, 1, -2), (1, -2,1)\}\rangle 
$$
(if $u,v,w$ are the three vertices of $H$, then $\dv (f_v)=-2v+ w+u$).
Hence
$$
\Phi_H \cong \frac{\Z\oplus \Z}{\langle (1,1)(1,-2)\rangle }\cong \Z/3\Z.
$$
\end{example}

\subsection{Torelli problem for   graphs}

We have seen another description of the Jacobian of an algebraic curve,
namely the ``classical" one given in \eqref{jacCC}.  This  description also has its graph theoretic analogue in the following 
 definition due to Kotani-Sunada, \cite{KS}.  For a weightless graph $G$, its  {\it Jacobian torus} is defined as follows
\begin{equation}
 \label{Jact}
{\bf{Jac}}(G):=\Bigr( H_1(G, \RR) / H_1(G, \Z) ; ( \; , \;)  \Bigl)=\Bigr({\RR^{g}/ \Lambda}; ( \; , \;)  \Bigl)
\end{equation}
where $( \; , \;)$ denotes the bilinear form defined on $C_1(G,\RR)$ (and hence on $ H_1(G, \RR)$) by linearly extending 
$$ (e,e'):=\kappa_{e,e'},\quad  \forall e,e'\in E.$$
\begin{remark}
In \cite{KS}  the notation ``$\Alb(G)$" is used instead of ${\bf{Jac}}(G)$, and the above space is called the {\it Albanese torus} of $G$, whereas the name ``Jacobian torus", and the notation   $\Jac(G)$, is used for something else.
In this paper however, for consistency with the terminology for tropical Jacobians we shall introduce later,
we use the notation ${\bf{Jac}}(G)$ as   in \eqref{Jact}.
\end{remark}
Now we can ask the   ``Torelli" question: is it true that  two graphs having  isomorphic Jacobian tori  are isomorphic?

One sees   easily that the answer is no. It is  in fact clear from the definition that the Jacobian torus,   defined in terms of the cycle spaces of the graph,  should not change if some bridge  of the graph gets contracted
to a vertex. 

Therefore one should   refine the question. 
First, for any graph $G$ we denote by $G^{(2)}$ the graph obtained from $G$ by contracting every bridge 
   to a vertex (in graph-theoretic terms,  $G^{(2)}$  is ``2-edge-connected", which explains the notation).
   Notice that $G^{(2)}$ is well defined and  has the same genus as $G$.

Now we ask:
      is it true that if two graphs $G_1$ and $G_2$ have isomorphic Jacobian tori, then  $G_1^{(2)}\cong G_2^{(2)}$?

 The answer is, again, no. The reason now is   more subtle: pick two non-isomorphic bridgeless graphs, $G_1$ and $G_2$.
Suppose, as   in Figure 6 below,  that there exists a bijection between $E(G_1)$ and $E(G_2)$ which induces an isomorphism between the cycle spaces $H_1(G_1, \Z)$ and $H_1(G_2, \Z)$. In this case one says that  $G_1$ and $G_2$ are {\it cyclically equivalent} and writes  $G_1\equiv_{cyc}G_2$.
 Then it is not hard to see that ${\bf{Jac}}(G_1)\cong {\bf{Jac}}(G_2)$.

 \begin{figure}[h]
\begin{equation*}
\xymatrix@=.5pc{
&&&&*{\circ} \ar@{-}[rrr]\ar@{-} @/_.5pc/[ddd]\ar@{-} @/^.5pc/[ddd]  &&&*{\circ}\ar@{-} @/_.5pc/[ddd]\ar@{-} @/^.5pc/[ddd]   
  &&&&&&&&&*{\circ}\ar@{-} @/_.5pc/[dddlll]\ar@{-} @/^.5pc/[dddlll]\ar@{-} @/^.5pc/[dddrrr]\ar@{-} @/_.5pc/[dddrrr]  &&&&\\
 &&&&&&&&&&&&&\\
  &&&&&&&&&&&\equiv_{cyc}&&&&\\
  && &&\ar@{-}@(ul,dl)*{\circ}\ar@{-}[rrr]&&&*{\circ}  &&&&&&*{\circ}\ar@{-}[rrr]&&&*{\circ}\ar@{-}[rrr]\ar@{-}@(ul,ur)&&&*{\circ}&&&&&}
\end{equation*}
\caption{Cyclically equivalent graphs.}
\end{figure}
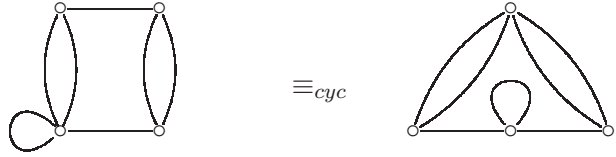

We mention    that the notion of cyclic equivalence we just defined, sometime called ``2-isomorphism",  has been given
by  Withney,
in \cite{whitney}, a constructive characterization   which is a helpful tool in proving the results we are presently describing.
   
We are ready to state an analogue of the Torelli theorem for graphs,   \cite[Thm 3.1.1] {CV1}.

\begin{thm}\label{torellig0} Let $G_1$ and $G_2$ be two  graphs. Then 
${\bf{Jac}}(G_1)  \cong {\bf{Jac}}(G_2)$ if and only if $ G_1^{(2)}\equiv_{\text{cyc}} G_2^{(2)}$.
 \end{thm}
 
    As we mentioned,   the ``if" part of the above theorem is easy.   

\

What about weighted graphs? 
We shall see in Theorem~\ref{torelligw} that the above theorem can be extended.
For the moment, we limit ourselves to explain how to generalize the basic definitions.

Following  \cite{BMV},
  the definition of cyclic equivalence extends trivially   as follows: two weighted graphs, $(G,w)$ and $(G',w')$, having the same genus are {\it cyclically equivalent} if so are
  $G$ and $G'$. 
  
Furthermore,  for a weighted graph ${\bf G}=(G,w)$ we can define the graph ${\bf G^{(2)}}=(G^{(2)}, w^{(2)})$ obtained by contracting every bridge of $G$ and by defining the weight function $w^{(2)}$ accordingly 
(recall that by definition of weighted contraction, whenever a bridge $e$ gets contracted to a vertex $v$,  the weight of $v$ is the sum of the weights of the ends of $e$). Again,  ${\bf G^{(2)}}$ has the same genus as ${\bf G}$.

\begin{example}
 
In the picture   we have a stable graph  ${\bf G}=(G,w)$ of genus $9$ on the left, the corresponding graph ${\bf G^{(2)}}=(G^{(2)}, w^{(2)})$ in the middle,
and, on the right, a graph   cyclically equivalent, but not isomorpic, to ${\bf G^{(2)}}$.  
 \begin{figure}[h]
\begin{equation*}
\xymatrix@=.5pc{
&&&&&&& &*{\bullet}\ar@{-} @/^.3pc/[dll]\ar@{-}[dd]^<{1}^>{1}&&&&&&
&&*{\bullet}\ar@{-} @/^.3pc/[dll]\ar@{-}[dd]^<{1}^>{1}&&*{\bullet}\ar@{-} @/_.3pc/[drr]\ar@{-}[dd]_<{2}_>{1}&&&&
\\
{\bf G}=&&\ar@{-}@(ul,dl)*{\circ}\ar@{-} @/^ 1.2pc/[rr] \ar@{-} @/_.2pc/[rr] \ar@{-} @/_1.1pc/[rr]  &&*{\bullet}\ar@{-}[rr]_<{1}_>{2}&&*{\bullet}  
&&&  &{\bf G^{(2)}}=&&
\ar@{-}@(ul,dl)*{\circ}\ar@{-} @/^ 1.2pc/[rr] \ar@{-} @/_.2pc/[rr] \ar@{-} @/_1.1pc/[rr]_>{3}  &&*{\bullet}&&&\equiv_{cyc}&&&
*{\circ}\ar@{-}@(ul,ur)\ar@{-} @/^ .8pc/[rr] \ar@{-} @/_.2pc/[rr] \ar@{-} @/_1.1pc/[rr]_>{2}  &&*{\bullet}
   \\ 
&&&&&&& &*{\bullet}\ar@{-} @/_.3pc/[ull]&&&&&
&&&*{\bullet}\ar@{-} @/_.3pc/[ull]&&*{\bullet}\ar@{-} @/^.3pc/[urr]&&&&&&&&&&&&&&&&\\
  }
\end{equation*}
\caption{}
\end{figure}\end{example}

    \subsection{Torelli problem for tropical curves}
  The  results we described in the previous subsection  were obtained around the first decade of the twenty-first century,
during a time when  tropical geometry was flourishing (see   \cite{MikhalkinICM} and \cite{MS} for more on the subject).
In particular the   Jacobian of a pure tropical curve was introduced and studied in \cite{MZ}
  in close analogy to   the Jacobian torus of a graph we introduced earlier.

Having just settled the Torelli problem for graphs, the question on how to extend it to metric graphs, i.e.  to tropical curves,
arose as a very natural one. This is the topic of the present subsection. 
  
  Let us start with the main definition.
   Following \cite{MZ}, the Jacobian of a pure tropical curve, $\Gamma= (G, \ell)$, is defined as the following polarized torus

\begin{equation}
\label{defjact}
{\bf{Jac}}(\Gamma) :=\Biggr({H_1(G, \RR) \over H_1(G, \Z)}; ( \; , \;)_{\ell}  \Biggl)
\end{equation}
where $( \; , \;)_\ell$ is the bilinear form defined on $C_1(G,\RR)$  by linearly extending 
$$ (e,e')_\ell:=\kappa_{e,e'}\ell(e),\quad  \forall e,e'\in E.$$
For simplicity, we postpone the definition of the Jacobian for an arbitrary tropical curve, and argue as in the previous subsection
to arrive at the statement of the Torelli theorem for pure tropical curves.

Let  $\Gamma_1=(G_1,\ell_1)$ and $\Gamma_2=(G_2,\ell_2)$ be two pure tropical curves.
We say that   $\Gamma_1$ and $\Gamma_2$ are {\it cyclically equivalent}, and write 
 $\Gamma_1\equiv_{\text{cyc}} \Gamma_2$, if 
   there is a length preserving bijection between $E(G_1)$ and $E(G_2)$ which induces an isomorphism between $H_1(G_1, \Z)$ and $H_1(G_2, \Z)$. Of course if two curves are cyclically equivalent, 
 so are the underlying graphs.
 
 Next, for a tropical curve $\Gamma=(G,\ell)$ we denote by $\Gamma^{(2)}=(G^{(2)},\ell^{(2)})$ the curve whose underlining graph, $G^{(2)}$, is defined as in the previous subsection and whose length function $\ell^{(2)}$ coincides with $\ell$ (as we have a natural inclusion
$E(G^2)\subset E(G)$).

Reasoning as in the previous subsection, we see that if $\Gamma_1^{(2)}\equiv_{\text{cyc}} \Gamma_2^{(2)}$
then ${\bf{Jac}}(\Gamma_1)\cong {\bf{Jac}}(\Gamma_2)$. Does the converse hold?

The answer is no. Consider the two tropical curves $\Gamma$ and $\Gamma'$ in the picture below
(ignoring the weights on the vertices for the moment) and assume that the ten unmarked edges (five on each curve) have  the same length,  say equal to  1. Now, $\Gamma$ and $\Gamma'$ are clearly not cyclically equivalent, but they turn out to have isomorphic Jacobians.
 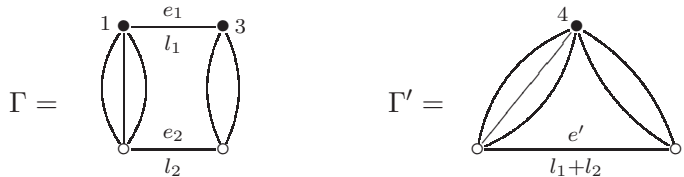
\begin{figure}[h]
\begin{equation*}
\xymatrix@=.5pc{
&&&&&&&&&&&&&&\\
&&&&*{\bullet} \ar@{-}[rrr]^{e_1}_{l_1}\ar@{-} @/_.7pc/[ddd]_<{1} \ar@{-} [ddd]\ar@{-} @/^.7pc/[ddd] &&&*{\bullet}\ar@{-} @/_.5pc/[ddd] \ar@{-} @/^.5pc/[ddd]^<{3}   
  &&&&&&&&&*{\bullet}\ar@{-}[dddlll]\ar@{-} @/_.7pc/[dddlll]_<{4} \ar@{-} @/^.7pc/[dddlll]\ar@{-} @/^.5pc/[dddrrr]\ar@{-} @/_.5pc/[dddrrr]  &&&&\\
 &&&&&&&&&&&&&\\
  &&\Gamma  =&&&&&&&&&&\Gamma' =&&&&\\
  && && *{\circ}\ar@{-}[rrr]^{e_2}_{l_2}&&&*{\circ}  &&&&&&*{\circ}\ar@{-}[rrrrrr]^{e'}_{l_1+l_2}&&&&&&*{\circ}&&&&&}
\end{equation*}
\caption{$\Gamma '=\Gamma^{(3)}$}
\end{figure}

In fact, observe that in $\Gamma$ the pair of edges $(e_1,e_2)$ has the property that
every cycle containing $e_1$ also contains $e_2$, and conversely. This property is equivalent to the fact that 
$(e_1,e_2)$ is a {\it separating pair} of edges, i.e. it disconnects $\Gamma$. Now,  there is
a bijection between the cycles in $\Gamma$ and those of $\Gamma'$ mapping bijectively cycles containing $e_1$ and $e_2$
 to cycles containing $e'$. Finally, and this is a key point, although this bijection 
does not preserve the number of edges in the cycles, it does
preserve  the total length of the cycles.

Let us turn this  example into a special case of something we have yet to define. For  a  tropical curve $\Gamma=(G,\ell)$ we introduce 
a curve $\Gamma^{(3)}=(G^{(3)},\ell^{(3)})$ defined as follows. First, we eliminate all the bridges
by   replacing $\Gamma$ with $ \Gamma^{(2)}$. Then for any  separating pair  of edges, $(e_1,e_2)$,
we contract one of them to a vertex  and set the length of the remaining edge equal to  
$\ell(e_1)+\ell(e_2)$. We repeat this process until   there are no separating pairs left;
the resulting curve is written $\Gamma^{(3)}=(G^{(3)},\ell^{(3)})$ (because the graph $G^{(3)}$ is
``3-edge connected").
Now,  the   curve  $\Gamma^{(3)}$   is not uniquely determined by $\Gamma$ but, as we will show in Lemma~\ref{Gamma3},  its cyclic equivalence class is.
 
For future purposes, we need   a definition.
 
\begin{defi}
\label{C1def}
Let $G=(E,V)$ be a  graph free from bridges.  
We define an equivalence relation on $E$ as follows: $e_1\sim e_2$ if $(e_1,e_2)$ is a separating pair,
i.e.  every cycle of $G$ containing $e_1$ contains $e_2$  and conversely.
The equivalence classes of this relation are called the {\it C1-sets} of $G$.
\end{defi}
It is clear that ``$\sim$"   is   an equivalence relation.

As we will explain   in Subsection~\ref{torellisec}, the terminology ``C1-set" is motivated by the corresponding algebro-geometric setting. 

The  C1-sets turn out to be elements of the poset $\SP_G$ defined earlier, in subsection ~\ref{CJsecc}.

 It is not hard to see that a
  bridgeless graph   is 3-edge connected if and only if  its C1-sets  have all cardinality $1$.
\begin{remark}
\label{C1rk}
Given a tropical curve $\Gamma=(G,w,\ell)$, the curve $\Gamma^{(3)}=(G^{(3)},w^{(3)},\ell^{(3)})$     can     be constructed as follows. For every C1-set of $G$, contract all but one of its edges.
Then set the length, $\ell^{(3)}(e)$, of the remaining edge, $e$,   equal to the sum of the lengths of all the edges in the C1-set.

The weight function $w^{(3)}$ is
  defined as for any weighted edge-contraction.
\end{remark}
\begin{example}
In Figure 7, the curve $\Gamma$ has only one C1-set with more than one element,
namely $\{e_1, e_2\}$. By contracting $e_1$ and defining  the length as in the picture
we get    $\Gamma^{(3)}=\Gamma'$.
\end{example}
In \cite[Thm 4.1.10]{CV1}  it is proved that two pure tropical curves $\Gamma_1$ and $\Gamma_2$ have isomorphic Jacobians if and only if $\Gamma_1^{(3)}$ and $\Gamma_2^{(3)}$ are cyclically equivalent.
This has been generalized to all tropical curves in \cite{BMV}. Before stating the theorem we need to 
 define, following \cite{BMV},   the Jacobian of a tropical curve $\Gamma= (G, w, \ell)$ of genus $g$, which  is   as follows
 
\begin{equation}
 \label{defjactw}
{\bf{Jac}}(\Gamma) :=\Biggr({H_1(G, \RR)\oplus \RR^{g-b_1(G)} \over H_1(G, \Z)\oplus \Z^{g-b_1(G)} }; ( \; , \;)_{\ell}  \Biggl)
\end{equation}
where $( \; , \;)_{\ell}$ is defined as in \eqref{defjact} on $H_1(G, \RR)$, and extended to $0$ on   $ \RR^{g-b_1(G)}$.
 
 Next, two tropical curves  $\Gamma_1=(G_1,w_1, \ell_1)$ and $\Gamma_2=(G_2,w_2, \ell_2)$
 are defined to be {\it cyclically equivalent} if so are the pure tropical curves $(G_1,\ell_1)$ and $(G_2,\ell_2)$.
 
 We can finally state  the tropical Torelli theorem, \cite[Thm 5.3.3]{BMV} (in different words).

\begin{thm}
\label{torellit} Let  $\Gamma_1$ and $\Gamma_2$ be  tropical curves. Then 
${\bf{Jac}}(\Gamma_1)\cong {\bf{Jac}}(\Gamma_2)$ if and only if $\Gamma_1^{(3)}\equiv_{\text{cyc}} \Gamma_2^{(3)}$.
 \end{thm}
    As a last comment on this theorem  we shall clarify  its statement  by proving that the
  cyclic equivalence
class of  $\Gamma^{(3)}$ is 
  uniquely determined by $\Gamma$.

\begin{lemma}
\label{Gamma3}
Let  $\Gamma=(G,w,\ell)$  be a   tropical curve. Then the cyclic equivalence class of    $\Gamma^{(3)}=(G^{(3)},w^{(3)},\ell^{(3)})$, defined in Remark~ \ref{C1rk}, is uniquely determined.
  \end{lemma}
\begin{proof}
Fix $\Gamma^{(3)}$ as above.
Since cyclic equivalence does not take  weights into account, we are free to ignore $w$ and $w^{(3)}$.

We shall   prove that there is a canonical  length-determining bijection between the edges of $G^{(3)}$ and the C1-sets of $G$
which, moreover,  induces a canonical bijection between the cycles of $G^{(3)}$  and those of $G$. 
 By definition of cyclic equivalence  this will conclude the proof.
 
By  Remark~\ref{C1rk} every edge, $e$, of $G^{(3)}$ corresponds to a unique C1-set of $G$ and every C1-set of $G$ is obtained in this way.  Next,  $\ell^{(3)}(e)$
equals the cardinality of this C1-set. This gives a canonical
 length-determining bijection between   edges of $G^{(3)}$ and  C1-sets of $G$.
 
By Definition~\ref{C1def},
 if a cycle of $G$ contains a certain  edge, then it contains the entire C1-set containing it. In other words, every cycle     is partitioned into C1-sets.
  Hence
 in the contraction  
 $ 
 G\to G^{(3)}
 $ 
 no cycle gets collapsed to a vertex, and different cycles of $G$  go to different    cycles  of $G^{(3)}$ (C1-sets are disjoint). We thus get a canonical bijection between the cycles of $G$  and those of $G^{(3)}$. Now we are done.
 \end{proof}

\subsection{Torelli problem for weighted graphs}
\label{twsec}
 We are now ready to deal with the Torelli problem for weighted graphs.

Let ${\bf{G}}=(G,w)$ be a weighted graph of genus $g$. Let its Jacobian torus be defined as follows
$$
{\bf{Jac}}({\bf{G}}) :=\Biggr({H_1(G, \RR)\oplus \RR^{g-b_1(G)} \over H_1(G, \Z)\oplus \Z^{g-b_1(G)} }; ( \; , \;)   \Biggl)
$$
 where $( \; , \;)$ is defined as for \eqref{defjactw}.

We shall now derive the following theorem   as a consequence of  Theorem~\ref{torellit}.
It could probably be proved directly, but, as far as we know, this has not been done.

\begin{thm}\label{torelligw} Let ${\bf{G_1}}$ and ${\bf{G_2}}$ be two  stable weighted graphs of genus $g$. Then 
${\bf{Jac}}({\bf{G_1}})  \cong {\bf{Jac}}({\bf{G_2}})$ if and only if ${\bf{G_1}}^{(2)}\equiv_{\text{cyc}} {\bf{G_2}}^{(2)}$.
 \end{thm}
\begin{proof}
We    prove the ``only if" part (the other part is clear); so assume ${\bf{Jac}}({\bf{G_1}})  \cong {\bf{Jac}}({\bf{G_2}})$.
By the other direction, we can assume ${\bf{G_1}}$ and ${\bf{G_2}}$  are free from bridges.

We introduce the locus in $\Mt_g$ of tropical curves all of whose edges have length $1$:
$$
\Mt_g[1]:=\{[\Gamma] =[(G,w,\ell)]\in \Mt_G: \  \ell(e)=1, \  \forall e\in E\}.
$$
We identify it with the set of stable graphs of genus $g$, as follows  
$$
{\rm S}_g \stackrel{\cong}{\la}\Mt_g[1]; \quad \quad {\bf{G}} =(G,w)\mapsto  [(G,w , \underline{1})]
$$
where $ \underline{1}$ denotes the constant length function equal to $1$.
Denote by $\Gamma_1$ and $\Gamma_2$ the tropical curves in $\Mt_g[1]$ corresponding to  ${\bf{G_1}}$ and ${\bf{G_2}}$;
as ${\bf{Jac}}({\bf{G_1}})  \cong {\bf{Jac}}({\bf{G_2}})$  we have, of course, ${\bf{Jac}}(\Gamma_1)\cong {\bf{Jac}}(\Gamma_2)$, and hence $\Gamma_1^{(3)}\equiv_{\text{cyc}} \Gamma_2^{(3)}$,  by Theorem~\ref{torellit}.

Let $i=1,2$; as  $\Gamma_i$ has all edges of length $1$, all the edges of   $\Gamma_i^{(3)}$ have  integer length.
Recall from Remark~\ref{C1rk} that the edges of  $\Gamma_i^{(3)}$ are in bijection with the C1-sets of $G_i$,
 and, under this bijection, the C1-set corresponding to an edge, $e$, of $\Gamma_i^{(3)}$ has cardinality $\ell^3(e)$.

Since $\Gamma_1^{(3)}$ and $\Gamma_2^{(3)}$ are cyclically equivalent, 
there is a length preserving bijection between the edges of $\Gamma_1^{(3)}$ and $\Gamma_2^{(3)}$,
  hence we
 obtain   a cardinality preserving bijection between the C1-sets of $G_1$ and those of $G_2$. By \cite[Prop. 2.3.9]{CV1} we conclude 
that $G_1$ and $G_2$ are cyclically equivalent,  and hence so are  ${\bf{G_1}}$ and ${\bf{G_2}}$. Our statement is proved.
\end{proof}

\subsection{Torelli problem for stable curves}
\label{torellisec}

We now go back to the set up of  subsection~\ref{CJsec}, and study the fibers of the 
Torelli map $ { \overline{\tau}}: \Mgb \la  \overline{A_g}$.
To simplify the exposition we shall assume our curves are  free from separating nodes
(i.e. their dual graph is bridgeless). As we have seen in Proposition~\ref{ctprop}, 
separating nodes don't   play a  significant role
but  some technicalities   are needed  to treat them, and the statement of     Theorem ~\ref{torellisc} needs to be modified;
see \cite[Thm. 2.1.4 ]{CV2}.  
 
Recall that   for a stable curve $X$  we set
$$
 { \overline{\tau}}([X])={\bf\overline{Jac}}(X) =\bigl[\Jac(X)\curvearrowright (\PX,\overline{\Theta(X)})\bigr]. 
$$ 
We shall  look more closely at the stratification of $\PX$.  By combining   algebro-geometric arguments with    the combinatorial tools we developed for the tropical Torelli problem  one can prove the following  facts.

\begin{enumerate}[(A)]

\item
\label{tA}
With respect to the stratification \eqref{PX}  and the support map \eqref{Supp},
  the codimension-one strata of  $\PX$ are in bijection with the C1-sets of $G_X$; cf. Definition~\ref{C1def}.
  (The terminology ``C1-set" stands for ``Codimension 1 set", and is motivated precisely by the present situation).

Hence for every C1-set, $S$, there is a unique stratum corresponding to it (in other words $|\Sigma_S|=1$); we shall denote by $P_S$ this stratum.

\item
\label{tB}
Let $P_S$ be the  the stratum corresponding to a C1-set, $S$. The 
 intersection  $\ov{\Theta(X)}\cap P_S$   determines the set of branch  points over $S$, i.e.  the set $\nu^{-1}(S)\subset X^{\nu}$ (abusing  notation as usual  viewing $S\subset \Sing(X)$).
\item
\label{tC}
For every C1-set $S$, 
the number of  irreducible components  of $\ov{\Theta(X)}\cap P_S$ equals  the cardinality of $S$.
 
 \item
 \label{tD}
 On the opposite side, in \eqref{PX} the maximum codimension of a stratum is equal to $b_1(G_X)$, and there is a unique such  ``smallest" stratum,
 $ P_{E}\subset   \PX$, where  $E=E(G_X)$,
  identified as follows  
$$
 P_{E}\cong \prod_{v\in V}\Pic ^{g_v-1}C_v^{\nu}.
$$
For any $v\in V$ denote by $\pi_v:  P_{E} \to \Pic ^{g_v-1}C_v^{\nu}$
the projection; then the intersection of  $\ov{\Theta(X)}$ with the smallest stratum  satisfies
$$
\pi_v(\ov{\Theta(X)}\cap  P_E) =\Theta(C_v^{\nu}).
$$
 
\end{enumerate}
 \
 
 Now, consider  two stable curves  $X_1$ and $X_2$,  such that ${\overline{\tau}}([X_1])={\overline{\tau}}([X_2])$;
how are they related? Our goal is to answer this question and show   they need not be isomorphic.

By hypothesis we have an isomorphism 
between $\Jac(X_1)\curvearrowright \ov{P_{g-1}(X_1)}$ and  $\Jac(X_2)\curvearrowright \ov{P_{g-1}(X_2)}$, hence, by Remark~\ref{orbit},
we have   a bijection between the  stratifications which preserves the dimension of the strata.  Hence, with the notation of Subsection~\ref{CJsec}, we have an isomorphism of posets  $\ST_{X_1}\cong \ST_{X_2}$.
Now, this turns out to imply that there is an isomorphism of posets $\SP_{X_1}\cong \SP_{X_2}$.
Using the combinatorial  set-up described in the previous sections and \cite[Thm 5.3.2]{CV1}, this implies
\begin{equation}
 \label{G123}
G_1^{(3)}\equiv_{cyc}G_2^{(3)}.
\end{equation}
In particular, we  have a bijection between   the C1-sets of $G_1$ and  $G_2$.

On the other hand, we have the
 isomorphism between the Theta divisors.
Hence for every  C1-set, $S_1$, of $G_1$ the intersection   ${\ov {\Theta(X_1)}}\cap P_{S_1}$ 
 is mapped isomorphically to 
 ${\ov {\Theta(X_2)}}\cap P_{S_2}$, where $S_2$ is the C1-set of $G_2$ corresponding to $S_1$. Part \eqref{tC} implies that  $S_1$ and $S_2$ have the same cardinality.
Combining this with  \eqref{G123} and arguing as   in the final step of the proof of Theorem~\ref{torelligw}, we conclude
$$
G_1 \equiv_{cyc}G_2 .
$$

The above relation implies that $X_1$ and $X_2$ have the same number of irreducible components, and the same number of nodes.

Now,  the isomorphism  $\overline{P_{g-1}(X_1)}\cong \overline{P_{g-1}(X_2)}$
maps the smallest stratum isomorphically to the smallest stratum.
By what we said in \eqref{tD},   we can apply the   Torelli theorem for smooth curves
to obtain  an isomorphism  
$$
X_1^{\nu}\smallsetminus \{\text{rational components}\}\cong X_2^{\nu}\smallsetminus \{\text{rational components}\}.
$$
On the other hand, we   already proved that $X_1$ and $X_2$    have the same number of components, 
hence  $X_1^{\nu}\ \cong X_2^{\nu}$.
Now, by \eqref{tB}, this isomorphism maps the branch points of $X_1$ to those of $X_2$.

Summarizing, we showed that
$X_1$ and $X_2$ have
\begin{enumerate}
 \item
 the same normalization, 
let us  denote it by   $Y$ and let $Y\stackrel{\nu_i}{\la}X_i$ be the normalization maps for $i=1,2$;
 \item
 cyclically equivalent graphs, and hence   the ``same" C1-sets;
 \item
 the same set   of branch points $B=\nu_1^{-1}(\Sing(X_1))=\nu_2^{-1}(\Sing(X_2))$. 
\end{enumerate}
Let us focus on  $B$, a set of points in $Y$ which must be pairwise glued
to form the nodes of $X_1$ and $X_2$. How are the points of $B$ to be glued? As we said, $B$ is determined together with a partition induced by the partition in C1-sets of $E(G_1)$ and $E(G_2)$. Indeed, 
let us identify 
the C1-sets of $G_1$ with those of  $G_2$  and let $S$ be one of them.
Since $S$ corresponds to a set of singular points,  $B$ contains a subset, $B_S$,
of $2|S|$ branch points mapping to $S$. As $S$ varies the sets $B_S$ form a partition of $B$.

The problem now is:  we do not know how to pairwise glue the   branch points in $B_S$,
unless $S$ has cardinality one, of course.
Hence
the curves $X_1$ and $X_2$ 
may fail to be isomorphic  if they have C1-sets of cardinality greater than one.
This is   the case in the following example.

\begin{example}
 Let $C_1$ and $C_2$ be two non-isomorphic smooth curves of genus  $2$,
 pick two distinct points on each, $p_i,q_i\in C_i$ which are not mapped to one another by an automorphism of $C_i$.
Set $Y=C_1\sqcup C_2$ and $B=\{p_1,p_2, q_1,q_2\}$
and let  the dual graph be  as in Figure 9  ($G$ is unique in its cyclic equivalence class). 
\begin{figure}[h]
\begin{equation*}
\xymatrix@=.5pc{
 G=& *{\bullet}\ar@{-} @/^ .8pc/[rrr]^{e_1}^>{2}^<{2} \ar@{-} @/_,8pc/[rrr]_{e_2}   &&&*{\bullet}  
   \\}
\end{equation*}
\caption{ }
\end{figure}
$G$ has only one C1-set, namely $\{e_1, e_2\}$, hence the partition induced on $B$ is the trivial one (only one set);
therefore we can form exactly two non isomorphic curves corresponding to our data, namely
 $$
 X_1= {{C_1\sqcup C_2}\over{p_1=p_2, q_1=q_2}} \quad {\text{ and }} \quad X_2= {{C_1\sqcup C_2}\over{p_1=q_2,q_1= p_2}}
 $$
 For such curves we have, indeed, 
  ${\bf\overline{Jac}}(X_1)\cong {\bf\overline{Jac}}(X_2)$.
 \end{example}

Let us contrast the above  example by the following  variation
\begin{example}
 Let $C_1$,  $C_2$ and $Y$ be as in the previous example,  
 pick three  distinct points, $p_i,q_i, r_i\in C_i$.
Set   $B=\{p_1,p_2, q_1,q_2, r_1,r_2\}$
and fix the following dual graph  (unique in its cyclic equivalence class)
\begin{figure}[h]
\begin{equation*}
\xymatrix@=.5pc{
 G=& *{\bullet}\ar@{-}[rrr]^{e_2}\ar@{-} @/^ 1.2pc/[rrr]^{e_1}^>{2}^<{2} \ar@{-} @/_,8pc/[rrr]_{e_3}   &&&*{\bullet}  
   \\}
\end{equation*}
\caption{ }
\end{figure}

Now 
$G$   has  three  C1-sets namely $\{e_1\}, \{ e_2\}, \{ e_3\}$ 
to each of which there corresponds a pair of points in $B$ which must, therefore, be glued together.
For example, if the partition induced on $B$ is $\{p_1,p_2\}, \{q_1,q_2\}, \{ r_1,r_2\}$  then the only stable curve corresponding to these data is 
 
 $$
 X = {{C_1\sqcup C_2}\over{p_1=q_1, p_2=q_2, r_2=r_2}}. $$
 \end{example}

 We are ready for the Torelli theorem.  

\begin{thm}
\label{torellisc}
 Let $X_1$ and $X_2$ be stable curves of genus at least $2$ without   separating nodes.   Then
 ${\bf\overline{Jac}}(X_1)\cong {\bf\overline{Jac}}(X_2)$
 if and only if the following holds.
 
\begin{enumerate}
\item
 There is an isomorphism $\phi:X_1^{\nu}\to X_2^{\nu}$.
 \item
There is a bijection between the C1-sets of $X_1$ and those of $X_2$ such that 
for every C1-set, $S_1$, of $X_1$ and the corresponding C1-set, $S_2$, of $X_2$ we have
$\phi(\nu ^{-1}(S_1))=\nu ^{-1}(S_2)$.
\end{enumerate}
In particular, if the dual graphs of $X_1$ and $X_2$ are 3-edge connected, then 
${\bf\overline{Jac}}(X_1)\cong {\bf\overline{Jac}}(X_2)$
if and only if $X_1\cong X_2$.
\end{thm}

The last part of the theorem follows from what we observed right  before Remark~\ref{C1rk}.
We refer to \cite{CV2} for   more  general and detailed results on the fiber of the Torelli map $\ov{\tau}$.
\section{Conclusions}

As we discussed in the previous section, in the paper \cite{BMV} the authors prove the tropical Torelli theorem   and construct the moduli space for
tropical curves. Furthermore, they construct
  a moduli space for principally polarized tropical abelian varieties which we have yet to introduce.
In fact, as its title indicates, the principal goal of that paper was the definition of the tropical Torelli map, in analogy with the classical Torelli map $\tau:M_g\to A_g$. In order to do that, they had to   construct both the moduli space for tropical curves and the moduli space for tropical abelian varieties.

A  principally polarized tropical abelian variety of dimension $g$ is defined as a polarized torus 
$$
\Bigr({\RR^{g } \over \Lambda }; ( \; , \;)   \Bigl),
$$
where $\Lambda$ is a lattice of maximal rank in $\RR^g$ and  $( ,  )$ a bilinear form on $\RR^{g }$ defining a semi-definite quadratic form whose null space admits a basis in $\Lambda\otimes \QQ$. The Jacobian of a tropical curve defined in \eqref{defjactw} is a principally polarized abelian variety.

We here  denote by  $\Agt$ the moduli space of tropical abelian varieties constructed in  \cite{BMV}, where a different notation is used.   $\Agt$ is  a topological space (proved to be Hausdorff in \cite{Chan}) and   a stacky fan,  and   has  remarkable combinatorial properties.
The authors also construct  and study the   tropical Torelli map,
$$
\tau^{\trop}:\Mgt \la \Agt;\quad \quad \Gamma \mapsto   {\bf{Jac}}(\Gamma);
$$
see also in  \cite{Chan},  \cite {CMV} and \cite{viviani}. 

Our   goal now is to summarize, through the   commutative diagram below,  the connections between  algebraic and tropical moduli spaces   surveyed in this paper.   

   $$\xymatrix@=.5pc{
  &&&& &&\Mgban \ar [rrrrrrr]^{\trop}  \ar[ddddr] \ar [ddddddddd]_{{\overline{\tau}}^\an}&&&&&&&\;\Mgtb&&\\
  &&&&&&&&&   && \\
  &&&&&&&&&&\ocM_g(K)\ar[rrruu]^{\trop_K}\ar[llldd]_{\red_K}& \\
  &&&&&&&&&   && \\
  && &&  &&&\Mgb\ar [ddd]^{\overline{\tau}} &&&\cM_g(K)\ar@{^{(}->}[uu] \ar[rrr]^{\trop_K}\ar[lll]_{\red_K}\ar [ddd]^{{\tau}_K} &&&\;\Mgt\ar@{^{(}->}[uuuu] \ar [ddd]^{\tau^\trop}   &&&&  \\
 && && &&&  && &&&& &&\\
 && &&  &&&  && &&&& &&\\
  &&&&&&&\overline{A_g} &&&\cA_g(K)\ar[rrr]^{\trop_K^{A_g}}\ar[lll]_{\red_K^{A_g}}&&&\;\Agt\ar@{.>}[dd] &&  
  &&&&&  && &&&& &&\\
  && && &&&  && &&&& &&\\
  && && &&\overline{A_g}^\an\ar [ruu]\ar@{.>} [rrrrrrr]&&&&&&&{\boxed{?}}&&\\
}$$

\

We have already encountered most of the maps appearing in the diagram.
Only the left vertical arrow,
${\overline{\tau}}^\an: \Mgban\la \overline{A_g}^\an$,
 is really  new.  The space ${\overline{A_g}}^\an$  is the Berkovich analytification of the  main irreducible component of  compactified moduli space of principally polarized abelian varieties, introduced in Subsection~\ref{CJsec}, and 
${\overline{\tau}}^\an$
 is the ``analytic" Torelli map, defined as the map of Berkovich analytic spaces associated to the Torelli map 
${\overline{\tau}}$; see \cite[Sect. 3.3.]{berkovich}.

Next, recall that $\cM_g(K)$ and $\cA_g(K)$ are the sets of, respectively, smooth curves of genus $g$ over $K$ and   principally polarized abelian varieties of dimension $g$ over $K$; then   ${{\tau}_K}$
is the corresponding Torelli map.

The map $\red_K$ is defined in Remark~\ref{Cs},  the map $\red_K^{A_g}$  is defined in Remark~\ref{redAK}, 
and the map $\trop_K$ is defined in \eqref{tropl}.

The square diagram in the bottom right  is a straightforward generalization of \cite[Thm. 4.1.7]{viviani}.

As the diagram illustrates, we don't known how to fill in its right-bottom corner. 
This would amount to constructing a compactification of the moduli space of tropical abelian varieties 
which could be identified with the tropicalization of $\overline{A_g}$. In this way  we could try and complete the picture by gluing together
the local reduction and tropicalization maps, as done for stable curves via the moduli space of extended tropical curves.

The same question could be asked by focusing on Jacobians  and by replacing, in the diagram, the moduli spaces of abelian varieties by the Schottky loci.  In this   special case also the problem is awaiting to be solved. 

\bibliographystyle{amsalpha}
\bibliography{AMS2015final}

\end{document}